%% file: bilevel_general2.tex
\documentclass{jnsao}
\usepackage[utf8]{inputenc}
\usepackage[finnish,english]{babel}
\usepackage[svgnames]{xcolor}
\usepackage[textsize=footnotesize,color=DarkOrange!40]{todonotes}
\usepackage{tikz, pgfplots}
\usepgfplotslibrary{colorbrewer}
\usepackage{enumitem}
\usepackage{algpseudocode}
\usepackage[nameinlink,capitalise]{cleveref}
\usepackage{nicematrix}
\usepackage{booktabs}
\usepackage{comment}
\usepackage{changes-simple}

\definecolor{hrefcolor}{rgb}{0.0,0.4,0.7}
\definecolor{citecolor}{rgb}{0.0,0.35,0.2}
\definecolor{structure}{rgb}{0.09,0.09,0.44}
\hypersetup{
    colorlinks=true,
    linkcolor=citecolor,
    citecolor=citecolor,
    filecolor=hrefcolor,
    urlcolor=hrefcolor,
    pdfencoding=auto,
    hypertexnames=false,
}

\newtheorem{assumption}[definition]{Assumption}
\counterwithin{assumption}{section}

\crefname{assumption}{Assumption}{Assumptions}

\manuscripteprint{2408.08123}
\manuscripteprinttype{arxiv}
\manuscriptdoi{10.46298/jnsao-2025-15577}
\manuscriptsubmitted{2025-04-25}
\manuscriptaccepted{2025-08-01}
\manuscriptvolume{5}
\manuscriptnumber{15577}
\manuscriptyear{2025}

\usepackage{float}

\usepackage{xcolor,cancel}

\floatstyle{ruled}
\floatname{algorithm}{Algorithm}
\newfloat{algorithm}{tbp!}{loa}
\numberwithin{algorithm}{section}

\newcommand{\adjointGeneral}{G}

\makeatletter
\InputIfFileExists{localsetup.tex}{}{}
\providecommand{\replybox}[1]{{#1}}
\makeatother
\def\tvreply#1{\replybox{\color{black}\textbf{#1}}}

\tikzset{notestyleraw/.append style={align=justify}}

\title{Single-loop methods for bilevel parameter learning in inverse imaging}
\shorttitle{Single-loop methods for bilevel parameter learning}

\author{%
    Ensio Suonperä\thanks{Department of Mathematics and Statistics, University of Helsinki, Finland. \mbox{\email{ensio.suonpera@helsinki.fi}},
    \orcid{0000-0002-2111-0239}}
    \and
    Tuomo Valkonen\thanks{MODEMAT Research Center in Mathematical Modeling and Optimization, Quito, Ecuador \emph{and} Department of Mathematics and Statistics, University of Helsinki, Finland. \email{tuomo.valkonen@iki.fi}, \orcid{0000-0001-6683-3572}}
}
\shortauthor{Suonperä and Valkonen}

\date{2024-08-15 (revised 2025-08-04)}

\acknowledgements{%
    This research has been supported by the Academy of Finland grants 314701, 320022, and 345486, and the Vilho, Yrjö and Kalle Väisälä Foundation grant for PhD studies
}

\input{symbols}

\begin{document}

\maketitle

\begin{abstract}
    Bilevel optimisation is used in inverse imaging problems for hyperparameter learning/identification and experimental design, for instance, to find optimal regularisation parameters and forward operators.  However, computationally, the process is costly.
    To reduce this cost, recently so-called single-loop  approaches have been introduced.  On each step of an outer optimisation method, they take just a single gradient step towards the solution of the inner problem.
    In this paper, we flexibilise the inner algorithm to include standard methods in inverse imaging.
    Moreover, as we have recently shown, significant performance improvements can be obtained in PDE-constrained optimisation by interweaving the steps of conventional iterative linear system solvers with the optimisation method.
    We now demonstrate how the adjoint equation in bilevel problems can also benefit from such interweaving.  We evaluate the performance of our approach on identifying the deconvolution kernel for image deblurring, and the subsampling operator for magnetic resonance imaging (MRI).
\end{abstract}

\section{Introduction}

Bilevel optimisation has recently found significant interest in inverse imaging problems, to learn or identify regularisation parameters and kernels \cite{delosreyes2014learning,kunisch2012bilevel,tuomov-tgvlearn,calatroni2015bilevel,reyes2021optimality,hintermueller2017optimal,hintermuller2015bilevel,holler2018bilevel,ehrhardt2023optimal,chung2024efficient}, as well as for experimental design, for example, to design optimal sampling patterns for magnetic resonance imaging (MRI) \cite{sherry2020learning}.
For our purposes, the bilevel problem involves solving
\begin{equation}
    \label{eq:intro:bilevel-problem}
    \min_{\alpha\in\AlphaSpace}~J(S_u(\alpha)) + R(\alpha)
    \quad\text{with}\quad
    S_u(\alpha) \in \argmin_{u\in U} F(u; \alpha)
\end{equation}
in Hilbert spaces $\AlphaSpace$ and $U$. Here $F$ would be an optimisation formulation of the inverse problem of interest, and $\alpha$ the aforementioned parameters that we want to choose optimally. For Tikhonov-type formulations
\begin{equation}
    \label{eq:inner-objective}
    F(u; \alpha) = f(u; \alpha) + g(Ku; \alpha),
\end{equation}
where $f$ is the data fidelity term, and $g \circ K$ forms, for example, a total variation regularisation term.
For instance, to identify the parameterisation of a forward operator $A_\alpha$, including the kernel of a convolution operator, or the subsampling weights of a subsampled Fourier or Radon transform as in \cite{sherry2020learning}, we could solve the problem
\[
    \min_{\alpha \in \R^n}~ \sum_{i=1}^m \frac{1}{2}\norm{u_i-z_i}^2 + \beta\norm{\alpha}_1
    \quad\text{where each}\quad
    u_i \in \argmin_u \frac{1}{2}\norm{A_\alpha u - b_i}^2 + \lambda \mathrm{TV}(u),
\]
where $\alpha$ is the parameterisation to be identified;  $u_i = [S_u(\alpha)]_i$ are the reconstructed images corresponding to the corrupted data $b_i$; and $z_i$ are the ground-truth images. The outer regularisation parameter $\beta>0$ controls the sparsity of the parameters-to-be-identified $\alpha$. This regularisation could be combined with a positivity constraint. In this example, we also fix the inner regularisation parameter $\lambda>0$, although it could also be identified.

However, bilevel optimisation can be very expensive. At the outset, it requires solving the inner problem $\min F_u$---in our domain of interest, a possibly expensive inverse problem itself---several times to find the desired optimal parameters.
Therefore, in recent years, there has been a growing interest---also from machine learning, where bilevel optimisation has relevance to adversary learning \cite{liu2021investigating}---in developing “single-loop” methods, that on each step of an algorithm to solve the outer problem \eqref{eq:intro:bilevel-problem}, only take a \emph{single step} of a conventional optimisation method towards a minimiser of the inner objective \eqref{eq:inner-objective}  \cite{suonpera2022bilevel,chen2021single,hong2020two,li2022fully,yang2021provably,dagreou2022framework}. There are also similar methods that take more than one but only a fixed (small) number of inner iterations for each outer iteration  \cite{ji2021bilevel,ji2022lower,ghadimi2018approximation,kwon2023fully}, while \cite{salehi2024maid,ehrhardt2021inexact} take an adaptive number of inner iterations.

In \cite{jensen2022nonsmooth}, we introduced the single-loop approach to nonsmooth PDE-constrained optimisation, such as total-variation regularised electrical impedance tomography (EIT) reconstruction, avoiding solving the PDE on each step of the optimisation method, and instead taking a single step of a conventional linear system solver (Jacobi, Gauss–Seidel, conjugate gradients) on each step of the optimisation method. This achieved significant speedups.

We now want to apply such single-loop approaches to the adjoint equations that also surface in bilevel optimisation.
Moreover, compared to the aforementioned works on single-loop bilevel optimisation methods, that generally use gradient descent for the inner problem, we want to use methods that are applicable to inverse problems with nonsmooth total variation regularisation. Gradient descent only applies to smooth problems. The next step from gradient descent is forward-backward splitting, which is applicable to simple total variation regularised denoising through a dual formulation. However, when both the data term and the regularisation term involve a difficult operator, neither forward-backward is applicable \cite{clason2020introduction}.
A popular choice then is the primal-dual proximal splitting method (PDPS) of Chambolle and Pock \cite{chambolle2010first}.
We want to use (single steps of) this method for the inner problem, yet remain flexible
towards other alternatives as well.

That said, due to inherent difficulties in bilevel optimisation, in this work we cannot yet allow for a nonsmooth inner problem.
Indeed, due to the use of the adjoint equation, we require the inner problem to be twice differentiable. Some recent works \cite{kwon2023fully,liu2022bome} avoid second-order derivatives in the algorithm through the value function reformulation, but, nevertheless, require significant second-order and other differentiability assumptions for the convergence theory.\footnote{The Polyak–Łojasiewicz inequality is also, essentially, a second-order growth assumption \cite{rebjock2024fast}.}
This may impose a performance penalty without a corresponding benefit in applications where the (nevertheless assumed) second-order information can easily be computed, and---this is one of our principal contributions---efficiently exploited.
These and, indeed, most algorithms in the literature, also do not allow a nonsmooth $R$ in the outer problem, as our approach does.
The convergence results are, moreover, for the gradient, while we are interested in iterate convergence.

In \cref{sec:abstract} we, therefore, introduce and prove the convergence of an abstract tracking approach to the solution of bilevel optimisation problems, based on arbitrary solvers for the inner problem and the adjoint equation that are subject to simple tracking estimates.
This abstract analysis significantly simplifies upon our earlier more specific analysis in \cite{suonpera2022bilevel}. Parts of the analysis we relegate to the appendix: integral to our convergence proofs, we recall and adapt to our needs a three-point monotonicity estimate from \cite{tuomov-proxtest} in \cref{sec:monotonicity}. Not integral to the theory, but helpful to verify the its conditions, we also discuss the regularity of the solution mapping of the inner problem in \cref{sec:lipschitz}.

We treat tracking estimates for the inner problem in \cref{sec:inner}, proving that they hold for both forward-backward splitting and the PDPS.
Correspondingly, in \cref{sec:adjoint}, we then prove that standard operator splitting schemes, such as Jacobi and Gauss–Seidel splitting, satisfy the adjoint tracking property.
We finish in \cref{sec:numerical} with numerical experiments: identifying a convolution kernel for image deconvolution, and an optimal subsampling operator for MRI.
We demonstrate, in particular, significant performance improvements from a block-Gauss–Seidel scheme for the adjoint equation.
We delay the derivation of details of the numerical realisation to \cref{sec:prox,sec:block-gs}.

\subsection*{Notation and basic concepts}
\label{par:notation}

We write $\linear(X; Y)$ for the space of bounded linear operators between the normed spaces $X$ and $Y$, and $\Id$ for the identity operator.
Generally $X$ will be Hilbert, so we can identify it with the dual $X^*$.
We use the notation $A > B$ (resp.~$A \ge B$) to indicate that $A-B$ is positive (semi-)definite
We write $\iprod{x}{y}$ for an inner product, and $B(x,r)$ for a closed ball in a relevant norm $\norm{\freevar}$.
For self-adjoint positive semi-definite $Q\in \linear(X; X)$ we write $\norm{x}_{Q} \defeq \sqrt{\iprod{x}{x}_{Q}} \defeq \sqrt{\iprod{Qx}{x}}$ and $B_Q(x,r)$ for a closed ball in the norm $\norm{\freevar}_Q.$
For operators $p \in \linear(X; \AlphaSpace)$, we set $\norm{p}_Q \defeq \norm{p Q^{1/2}}_{\linear(X; \AlphaSpace)}$.
Pythagoras' \emph{three-point identity} in Hilbert spaces then states
\begin{equation}\label{3-point-identity}
    \iprod{x-y}{x-z}_Q = \frac{1}{2}\norm{x-y}^2_Q - \frac{1}{2}\norm{y-z}^2_Q + \frac{1}{2}\norm{x-z}^2_Q \qquad \text{for all } x,y,z\in X.
\end{equation}
We also extensively use Young's inequality
\[
    \iprod{x}{y}_Q \leq \frac{a}{2}\norm{x}_Q^2 + \frac{1}{2a}\norm{y}_Q^2 \qquad \text{for all } x,y\in X,\, a > 0.
\]

For $G \in C^1(X)$, we write $G'(x) \in X^*$ for the Fr\'{e}chet derivative at $x$, and $\grad G(x) \in X$ for its Riesz presentation, i.e., the gradient.
For $E \in C^1(X; Y)$, since $E'(x) \in \linear(X; Y)$, we use the Hilbert adjoint to define $\grad E(x) \defeq E'(x)^* \in \linear(Y; X)$.
Then the Hessian $\grad^2 G(x) \defeq \grad[\grad G](x) \in \linear(X; X)$.
When necessary we indicate the differentiation variable with a subscript, e.g., $\grad_u F(u, \alpha)$.

We define $\extR := \R \cup \{\infty\}$ with the usual arithmetic on $\R$ extended by $t+\infty = \infty$ for all $t\in\R.$
For convex $R: X \to \extR$, we write $\Dom R$ for the effective domain and $\subdiff R(x)$ for the subdifferential at $x$.
With slight abuse of notation, we identify $\subdiff R(x)$ with the set of Riesz presentations of its elements.
We define the \term{proximal map} as $\prox_R(x) \defeq \argmin_z \frac{1}{2}\norm{z-x}^2 + R(z)=\inv{(\Id+\subdiff R)}(x)$.
For $F:X\to \extR,$ we define $F^*:X^*\to \extR$ for \emph{Fenchel conjugate} (or \emph{convex conjugate}) of $F$ as
$
	F^*(x^*) = \sup_{x\in X}\{\iprod{x^*}{x} - F(x)\}.
$

\section{An abstract tracking approach}
\label{sec:abstract}

In this section, we prove the convergence of our proposed abstract tracking approach to bilevel optimisation.
We start in \cref{subsec:problem} by describing the problem, and then deriving in \cref{subsec:optimality cond} the optimality conditions that we try to solve; in particular, the adjoint equation.
We then introduce and discuss in \cref{sec:algorithm} the abstract algorithm.
We continue by stating the assumptions of the abstract method in \cref{sec:assumptions}, and then proving its convergence based on these assumptions \ref{sec:convergence}. The assumptions in particular involve inner problem and adjoint tracking conditions. We verify these for several explicit algorithms in the coming \cref{sec:inner,sec:adjoint}.

\subsection{Problem description}
\label{subsec:problem}

We slightly generalise the problem formulation \eqref{eq:intro:bilevel-problem}, and seek to solve
\begin{equation}
    \label{eq:bilevel-problem}
    \min_{\alpha\in\AlphaSpace}~J(S_u(\alpha)) + R(\alpha)
    \quad\text{subject to $S_u(\alpha) \in U$ satisfying}\quad
    0 = G(S_u(\alpha); \alpha),
\end{equation}
where, on Hilbert spaces $U$ and $\AlphaSpace$, $J: U \to \R$ is convex and Fréchet differentiable, $R: \AlphaSpace \to \extR$ is convex, proper, and lower semicontinuous, and $\adjointGeneral: U \times \AlphaSpace \to U$.
We call $S_u: \AlphaSpace \to U$ the solution mapping.
We will generally assume it to be well-defined (and single-valued) on the effective domain of $R$, i.e., in $\Dom R \defeq \{\alpha \in \AlphaSpace \mid R(\alpha) < \infty\}$.
Of all the function, we will also be making further assumptions, including differeniability assumptions, in \cref{sec:assumptions}

Taking $G=\grad_u F$ for sufficiently regular $F$, \eqref{eq:bilevel-problem} is a necessary condition for  \eqref{eq:intro:bilevel-problem}.
If $F$ is convex in $u$, this condition is also sufficient.
However, we want to allow for other optimality conditions, such as the primal-dual conditions that arise from the Frenchel–Rockafellar theorem, and that our used in the context of primal-dual methods (see, e.g., \cite{clason2020introduction}).
We therefore allow for abstract optimality conditions for the inner problem via a general $G$.

\subsection{Optimality conditions and the adjoint equation}
\label{subsec:optimality cond}

Assume that both $G$ and $S_u$ are Fréchet differentiable.
Suppose a solution $S_u(\alpha)$ exists for all $\alpha$ near  $\tilde\alpha \in \AlphaSpace$.
Then $\adjointGeneral(S_u(\alpha); \alpha)  = 0$ for all such $\alpha$, so by implicit differentiation
\[
    0=\grad_{\alpha}S_u(\alpha) \grad_{u} \adjointGeneral(S_u(\alpha), \alpha) + \grad_{\alpha} \adjointGeneral(S_u(\alpha), \alpha).
\]

That is, $p=\grad_{\alpha}S_u(\alpha)$ solves for $u=S_u(\alpha)$ the \term{adjoint equation}
\begin{equation}
    \label{eq:p-row-oc}
    0=p \grad_{u} \adjointGeneral(u, \alpha) + \grad_{\alpha} \adjointGeneral(u, \alpha).
\end{equation}
We introduce the corresponding solution mapping for the \term{adjoint variable} $p$,
\begin{equation}
    \label{def:S_p}
    S_p(u,\alpha) := - \grad_{\alpha} \adjointGeneral(u; \alpha) \left (\grad_u \adjointGeneral(u; \alpha)\right )^{-1}.
\end{equation}
We will later make assumptions that ensure that $S_p$ is well-defined. Then
$
    \grad S_u(\alpha) = S_p(S_u(\alpha), \alpha).
$

Since $S_u: \AlphaSpace \to U$, the Fréchet derivative $S_u'(\alpha) \in \linear(\AlphaSpace; U)$ and the Hilbert adjoint $\grad_\alpha S_u(\alpha) \in \linear(U; \AlphaSpace)$ for all $\alpha$. Consequently $p \in \linear(U; \AlphaSpace).$

By the sum rule for Clarke subdifferentials (denoted $\subdiff_C$) and their compatibility with convex subdifferentials and Fréchet differentiable functions \cite{clarke1990optimization}, we obtain
\[
\subdiff_C (J \circ S_u+R)(\opt{\alpha})
=
\grad_{\alpha}(J \circ S_u)(\opt{\alpha}) +  \partial R(\opt{\alpha})
=
\grad_{\alpha}S_u(\opt{\alpha})\grad_{u}J(S_u(\opt{\alpha})) +  \partial  R(\opt{\alpha}).
\]
The Fermat principle for Clarke or Mordukhovich subdifferentials then furnishes the necessary optimality condition
\begin{equation}
    \label{eq:outer-oc}
    0 \in \grad_{\alpha}(J \circ S_u)(\opt{\alpha}) +  \partial R(\opt{\alpha})= \grad_{\alpha}S_u(\opt{\alpha})\grad_{u}J(S_u(\opt{\alpha})) +  \partial  R(\opt{\alpha}).
\end{equation}

We combine the inner optimality condition $\adjointGeneral(S_u(\alpha); \alpha)  = 0$, the adjoint equation \eqref{eq:p-row-oc}, and the outer optimality condition \eqref{eq:outer-oc} as the inclusion
\begin{subequations}
\label{eq:main_optimality_condition}
\begin{gather}
    0 \in H(\opt{u}, \opt{p}, \opt{\alpha})
    \shortintertext{with}
    H(u,p,\alpha) := \begin{pmatrix}
        G(u; \alpha)  \\
        p \grad_{u} \adjointGeneral(u; \alpha) + \grad_{\alpha} \adjointGeneral(u; \alpha)  \\
        p\grad_{u}J(u) +  \partial R(\alpha)
    \end{pmatrix}
    \quad\text{for all}\quad
    u\in U, p\in \linear(U; \AlphaSpace), \alpha\in\AlphaSpace.
\end{gather}
\end{subequations}
This is the optimality condition that our proposed methods attempt to satisfy.

\subsection{Algorithm}
\label{sec:algorithm}

As already discussed, we generally assume that there exist solution mappings $S_u$ and $S_p$ of the inner problem and the adjoint equation, i.e., the first two lines of the inclusion \eqref{eq:main_optimality_condition}.
However, our algorithms do not attempt to solve these computationally expensive equations exactly on each step.  Instead, in  the abstract \cref{alg:abstract}, we assume to be given an abstract inner algorithm and an abstract adjoint algorithm, that satisfy a corresponding \term{inner tracking} and \term{adjoint tracking} property (\cref{step:abstract-alg:inner,step:abstract-alg:adjoint} of the algorithm). The idea is to implement those steps  by taking a single step of a conventional optimisation method (\cref{sec:inner}), or a linear system solver (\cref{sec:adjoint}). The corresponding algorithm defines the tracking parameters, that are never to be explicitly used; they are merely needed for the convergence theory.

\begin{algorithm}
    \caption{Bilevel abstract tracking approach (BATA)}
    \label{alg:abstract}
    \begin{algorithmic}[1]
        \Require
        Functions $J: U \to \R$, and $R: \AlphaSpace \to \extR$, with $J$ Fréchet differentiable and $R$ convex, on Hilbert spaces $U$ and $\AlphaSpace$.
        Outer step length $\sigma>0$.
        Norm-defining $Q \in \linear(U; U)$ and tracking parameters $\kappa_u,\kappa_p>1$; $\pi_u,\pi_p>0$; as well as $C_S>0$, satisfying the further bounds in \cref{ass:abstract:main}.

        \State Pick an initial iterate $(u^{0}, p^{0}, \alpha^{0}) \in U\times \linear(U; \AlphaSpace) \times \AlphaSpace.$
               Set $\alpha^{-1} \defeq \alpha^0$.
        \For{$k \in \N$}
        \State\label{step:abstract-alg:inner}
        Find (with an algorithm from \cref{sec:inner}) a $\nexxt u \in U$ satisfying the inner tracking property
        \begin{equation*}
            \kappa_u\norm{\nexxt u - S_u(\this\alpha)}_Q
            \leq
            \norm{\this u - S_u(\prev\alpha)}_Q
            +
            \pi_u \norm{\this\alpha - \prev\alpha}.
        \end{equation*}
        \State\label{step:abstract-alg:adjoint}
        Find (with an algorithm from \cref{sec:adjoint}) a $\nexxt p \in \linear(U; \AlphaSpace)$ satisfying the adjoint tracking property
        \begin{equation*}
            \kappa_p\norm{\nexxt p - \grad_{\alpha}S_u(\this\alpha)}_{\inv Q}
            \leq
            \norm{\this p - \grad_{\alpha}S_u(\this\alpha)}_{\inv Q}
            +
            C_S\norm{\nexxt u - S_u(\this\alpha)}_Q
            +
            \pi_p \norm{\this\alpha - \prev\alpha}.
        \end{equation*}
        \State Update $\alpha^{k+1} \defeq \prox_{ \sigma R}\left (\alpha^{k} - \sigma  p^{k+1}\grad_{u}J(u^{k+1})\right )$.
        \EndFor
    \end{algorithmic}
\end{algorithm}

\subsection{Assumptions}
\label{sec:assumptions}

We next state essential structural, initialisation, and step length assumptions.
We start with a contractivity condition needed for the proximal step with respect to $R$.

\begin{assumption}
    \label{ass:contractivity}
    Let $R: \AlphaSpace \to \extR$ be convex, proper, and lower semicontinuous.
    We say that $R$ is \term{locally prox-$\sigma$-contractive at $\opt\alpha \in \AlphaSpace$ for $q \in \AlphaSpace$} (within $A \subset \Dom R$) if there exist $\sigma, C_R > 0$ and a neighbourhood $A \subset \Dom R$ of $\opt\alpha$ such that, for all $\alpha \in A$,
    \[
    \norm{D_{\sigma R}(\alpha)-D_{\sigma R}(\opt\alpha)} \le \sigma C_R \norm{\alpha-\opt\alpha}
    \quad\text{for}\quad
    D_{\sigma R}(\alpha) \defeq \prox_{\sigma R}(\alpha - \sigma q)-\alpha.
    \]
    If $\sigma>0$ can be arbitrary with the same factor $C_R$, we drop the word “locally”.
\end{assumption}

The assumption holds for smooth functions \cite[Theorem A.4]{suonpera2022bilevel}. It also holds for indicator functions of convex sets if $q=0$ \cite[Theorem A.2]{suonpera2022bilevel}, and $R=\beta\norm{\freevar}_1 + \delta_{[0, \infty)^n}$ if $-q = (\beta, \ldots, \beta) \in \subdiff R(\opt\alpha)$  \cite[Theorem A.1]{suonpera2022bilevel}.
When applying the assumption to  $\opt\alpha$ satisfying \eqref{eq:main_optimality_condition}, we will take $-q = -\opt p\grad_u J(\opt u) \in \subdiff R(\opt\alpha)$.
The restriction on $q$ in the two nonsmooth examples thus serves to forbid strict complementarity: optimal solutions cannot involve interior subdifferentials.
Intuitively, this restriction serves to forbid the \emph{finite identification} property \cite{hare2004identifying} of proximal-type methods, as $\{\alpha^n\}$ cannot converge too fast in our current proof techniques for the stability of the inner problem and adjoint with respect to perturbations of $\alpha$.

We now come to our main assumption for the abstract algorithm.

\begin{assumption}
    \label{ass:abstract:main}
    Let $U$ and $\AlphaSpace$ be a Hilbert spaces. Let $R: \AlphaSpace \to \extR$ be convex, proper, and lower semicontinuous, and $J: U \to \R $ be convex and Fréchet differentiable.
    Also let $G: U  \times \AlphaSpace \to U$.
    Pick $(\opt u,\opt p,\opt \alpha) \in H^{-1}(0)$ and let $\{(u^n, p^n, \alpha^n)\}_{n\in \mathbb{N}}$ be generated by \cref{alg:abstract} for a given initial iterate $(u^{0}, p^{0}, \alpha^{0}) \in U \times \linear(U; \AlphaSpace) \times \Dom R$.
    For some positive $r_{\alpha}, r_u > 0$, we then suppose that
    \begin{enumerate}[label=(\roman*)]
        \item \label{item:abstract:main:solution-map-and-F}
        There exists an inner problem solution mapping $S_u \in C^1(\AlphaSpace_{2r}; U)$, satisfying $G(S_u(\alpha);\alpha) = 0$ for all $\alpha \in \AlphaSpace_{2r} \defeq B(\opt{\alpha}, 2r_{\alpha}) \isect \Dom R$. For a positive definite self-adjoint operator $Q \in \linear(U; U)$, this mapping is also $L_{S_u}$-Lipschitz w.r.t. the $Q$-norm, i.e. $\norm{S_u(\alpha_1) - S_u(\alpha_2)}_{Q} \leq L_{S_u} \norm{\alpha_1 - \alpha_2}$ for $\alpha_1,\alpha_2 \in  \AlphaSpace_{2r}$.

        \item \label{item:abstract:main:inner-tracking}
        We are given an \term{inner algorithm} $A_u$ that satisfies the inner tracking property
        \begin{equation*}
            \kappa_u\norm{A_u(u, \alpha_2) - S_u(\alpha_2)}_Q
            \leq
            \norm{u - S_u(\alpha_1)}_Q
            +
            \pi_u \norm{\alpha_2 - \alpha_1}
        \end{equation*}
        for any $u\in B_Q(\opt u, r_u)$ and $\alpha_1,\alpha_2\in \AlphaSpace_{2r}$ with $\pi_u>0$ and $\kappa_u>1.$

        \item \label{item:abstract:main:adjoint:tracking}
        We are given an \term{adjoint algorithm} $A_p$ that satisfies the adjoint tracking property
        \begin{equation*}
            \kappa_p\norm{A_p(u, p, \alpha_2) - \grad_{\alpha}S_u(\alpha_2)}_{\inv Q}
            \leq
            \norm{p - \grad_{\alpha}S_u(\alpha_1)}_{\inv Q}
            +
            C_S\norm{u - S_u(\alpha_2)}_Q
            +
            \pi_p \norm{\alpha_2 - \alpha_1}
        \end{equation*}
        for any $p\in \linear(U; \AlphaSpace), u\in B_Q(\opt u, r_u)$ and $\alpha_1,\alpha_2\in \AlphaSpace_{2r}$ with $C_S, \pi_p>0$ and $\kappa_p>1.$

        \item \label{item:abstract:main:outer-objective}
        The outer fitness function $J$ is Lipschitz continuously differentiable with factor $L_{\grad J}$, and the function $J\circ S_u$ is $\gamma_\alpha$-strongly convex and $L_\alpha$-Lipschitz differentiable in $B(\opt{\alpha}, r) \isect \Dom R$ for some $\gamma_\alpha, L_\alpha > 0$.
        Moreover, $R$ is locally prox-$\sigma$-contractive at $\opt\alpha$ for $\opt p\, \grad_u J(\opt u)$ within $B(\opt{\alpha}, r_{\alpha}) \isect \Dom R$ for some $C_R\ge 0$.

        \item
        \label{item:abstract:main:testing-params}
        The relative initialization bounds
        \[
            \norm{u^{1}-S_u(\alpha^{0})}_{Q} \leq C_u \norm{\alpha^{0} - \opt{\alpha}}
            \quad\text{and}\quad
            \norm{p^{1}-\grad_{\alpha}S_u(\alpha^{0})}_{\inv Q} \leq C_p \norm{\alpha^{0} - \opt{\alpha}}
        \]
        hold with constants $C_u,C_p>0,$ which satisfy
        \begin{gather*}
        	C_S\inv\kappa_uC_u < (\kappa_p-1)C_p
        	\quad
        	\text{ and }
        	\quad
        	C_{\alpha} \defeq
            L_{\grad J}N_pC_u + N_{\grad J}C_p < \gamma_\alpha
        \end{gather*}
        for $N_p \defeq N_{\grad S_u} + C_p r_{\alpha}$,
        \[
            N_{\grad J}
            \defeq \max_{\alpha\in B(\opt{\alpha}, r_{\alpha})\isect \Dom R} \norm{\grad_u J (S_u(\alpha))}_{Q},
            \quad\text{and}\quad
            N_{\grad S_u} \defeq
            \max_{\alpha\in B(\opt{\alpha}, r_{\alpha})\isect \Dom R}
            \norm{\grad_{\alpha}S_u(\alpha)}_{\inv Q}.
        \]

		\item \label{item:abstract:main:local}
		The initial outer iterate $\alpha^0\in B(\opt \alpha, r)$ for $r = \min \{ r_{\alpha}, r_u/(C_u + L_{S_u})\}.$

        \item \label{item:abstract:main:outer-step-length}
        The outer step length $\sigma > 0$ satisfies
        \[
        \sigma \leq \frac{1}{C_\alpha + L_\alpha + C_R}
        \min\left\{\frac{(\kappa_u -1)C_u}{\pi_u +\kappa_u C_u}, \frac{(\kappa_p-1)C_p- C_S\inv\kappa_uC_u}{\pi_p + \kappa_pC_p + C_S\inv\kappa_u\pi_u} \right\}
        \]
        and
        \[
        	\sigma < (\gamma_\alpha - C_\alpha)\frac{2}{L_\alpha^2}.
        \]
    \end{enumerate}
\end{assumption}

\begin{remark}[Interpretation]
    \label{rem:abstract:interpretation}
    The condition \cref{item:abstract:main:solution-map-and-F} of \cref{ass:abstract:main} ensures that the inner problem solutions do not vary uncontrollably as the parameter $\alpha$ changes. This is necessary to ensure that small changes of one variable result in small changes in the other variables as well.
    We discuss the condition more in the next \cref{rem:GIFB:solution-map-existence}.

    The conditions \cref{item:abstract:main:inner-tracking,item:abstract:main:adjoint:tracking} likewise ensure that the specific algorithms that are used to generate inner and adjoint iterates, produce small steps in response to small changes in $\alpha$, and, with no change, are contractive. We verify these conditions for exemplary optimisation algorithms and linear system splitting schemes in \cref{sec:inner,sec:adjoint}.

    The first part of \cref{item:abstract:main:outer-objective} is a second order growth and boundedness condition, standard in smooth optimisation.
    Together with \cref{item:abstract:main:solution-map-and-F}, the condition ensures that $\alpha \mapsto \grad_{\alpha}(J\circ S_u)(\alpha)$ is Lipschitz in $B(\opt{\alpha}, r)$.
    The local strong convexity of $J \circ S_{u}$ in the smooth case amounts to $\grad^2_{\alpha}(J\circ S_u)(\alpha) > \gamma_\alpha \Id$ holding locally.
    We refer to \cite[Remark 2.3]{suonpera2022bilevel} on how to reduce the latter to properties of $\grad_u^2 J(S(\alpha))$ and of $S$.
    The second part of \cref{item:abstract:main:outer-objective} essentially prevents $\prox_{\sigma R}$ from having a finite identification property, as discussed above.

    The condition \cref{item:abstract:main:testing-params} ensures that the initial inner problem and adjoint iterates are good \emph{relative} to the outer problem iterate.
    If $u^1$ solves the inner problem for $\alpha^0$, \cref{item:abstract:main:testing-params} holds for any $C_u>0$.
    Therefore, \cref{item:abstract:main:testing-params} can always be satisfied by solving the inner problem for $\alpha^0$ to high accuracy. This condition \emph{does not} require $\alpha^0$ to be close to a solution $\opt\alpha$ of the entire problem.
    We stress that \emph{the inequality on constants in \cref{item:abstract:main:testing-params} can always be satisfied by good relative initialisation (small $C_u,C_p>0$)}.
   	The constants $N_{\grad J}$ and $N_{\grad S_u}$ require analysing the specific inner problem: how large can the solutions $u=S_u(\alpha)$ for $\alpha$ in $B(\hat\alpha, r_\alpha)$? How fast do the vary? Practically, we will not know $r_\alpha$, so want to do this analysis for a large value.

    The semifinal \cref{item:abstract:main:local} is a standard initialisation condition for local convergence. If $r>0$ can be arbitrarily large, we obtain \emph{global convergence}.
    The final \cref{item:abstract:main:outer-step-length} is a step length restriction on the outer problem. Practically, it requires $\sigma>0$ to be sufficiently small.
\end{remark}

\begin{remark}[Existence and differentiability of the solution map]
    \label{rem:GIFB:solution-map-existence}
    The existence and differentiability of a $S_u(\alpha)=u$ such that $G(u, \alpha)=0$ needs to be explicitly verified.
    When $G=\grad_u F$, Fermat's principle proves that $\min_u F(u, \alpha)$ has a solution.
    That in turn, follows from lower semicontinuity and coercivity.

    Suppose then that $G$ is continuously differentiable in both variables, and that for some $u\in B(\opt{u}, r_u)$ and $\alpha\in B(\opt{\alpha}, 2r_{\alpha}) \isect \Dom R$ we have $G(u; \alpha)=0$ with $\grad_u G(u; \alpha)$ invertible.
    Then the implicit function theorem shows the existence of a unique continuously differentiable $S_u$ in a neighborhood of $\alpha$.
    Such an $S_u$ is also Lipschitz in a neighborhood of $\alpha$; see, e.g., \cite[Lemma 2.11]{clason2020introduction}.
    If $\AlphaSpace$ is finite-dimensional, a compactness argument gluing together the neighborhoods then proves the continuity and Lipschitz properties of \cref{ass:abstract:main}\,\cref{item:abstract:main:solution-map-and-F}.
    In \cref{sec:lipschitz} we discuss relaxations of these conditions to mere right-invertibility of $G_u(u; \alpha)$ (left-invertibility of $\grad_u G(u; \alpha)$).
\end{remark}

\begin{remark}[Inner and adjoint algorithms]
    Although for visual reasons \cref{ass:abstract:main} parametrises the inner and adjoint algorithms $A_u$ and $A_p$ by just $u$, $p$, and $\alpha$, which will in application be previous iterates, our proofs can easily handle iteration-dependent $A_u$ and $A_p$, and therefore dependence on the entire history of iterates.
\end{remark}

\subsection{Convergence analysis}
\label{sec:convergence}

We now prove the convergence of \cref{alg:abstract} subject to \cref{ass:abstract:main}.
Throughout, we take the assumption holding as a given, and \emph{use the constants from it}.
We also tacitly take it that $\alpha^n \in \Dom R$ for all $n \in \N$, as this is guaranteed by the assumptions for $n=0$, and by the proximal step in the algorithm for $n \ge 1$.

An essential goal here is to bound the error in the inner and adjoint iterates $u^{n+1}$ and $p^{n+1}$ in terms of the outer iterates $\alpha^n$ with
\begin{subequations}
	\label{grad:closeness_formulas}
	\begin{align}
		\label{grad:u_u_alpha_closeness_formula}
		\norm{u^{n+1}-S_u(\alpha^{n})}_{Q}
		&
		\leq
		C_u
		\norm{\alpha^{n} - \opt{\alpha}},
        \quad\text{and}
		\\
		\label{grad:adjoint_closeness_formula}
		\norm{p^{n+1}-\grad_{\alpha}S_u(\alpha^n)}_{\inv Q}
		&
		\leq
		C_p\norm{\alpha^n- \opt \alpha}
		.
	\end{align}
\end{subequations}
We prove by induction that \cref{ass:abstract:main} implies these bounds for each $n\in\N$.
We also derive bounds on the outer steps, and local monotonicity estimates.

\begin{lemma}
    \label{lemma:p-np}
    Suppose \cref{ass:abstract:main} holds.
    Let $n \in \N$ and suppose that \cref{grad:adjoint_closeness_formula} holds
    with $\alpha^n \in B(\opt{\alpha}, r)$.
    Then $\norm{p^{n+1}}_{\inv Q} \leq N_p$.
\end{lemma}
\begin{proof}
    We estimate using \eqref{grad:adjoint_closeness_formula} and the definitions of the relevant constants in \cref{ass:abstract:main} that
    \begin{align*}
        \norm{p^{n+1}}_{\inv Q}
        &
        \leq
        \norm{\grad_{\alpha}S_u(\alpha^n)}_{\inv Q} + \norm{p^{n+1}-\grad_{\alpha}S_u(\alpha^n)}_{\inv Q}
        \\
        &
        \leq
        N_{\grad S_u} + C_p \norm{\alpha^n- \opt \alpha} \leq N_{\grad S_u} + C_p r_{\alpha} = N_p.
        \qedhere
    \end{align*}
\end{proof}

The next lemma bounds the steps taken for the outer problem variable.

\begin{lemma}
    \label{lemma:alpha-option}
    Let $n \in \N$.
    Suppose \cref{ass:abstract:main} and \cref{grad:closeness_formulas} hold, and $\alpha^n\in B(\opt\alpha, r)$.
    Then
    \begin{gather}
        \label{ineq:alpha_in_exact_grad}
        \norm{\alpha^{n+1}-\alpha^n}
        \le
        \sigma [(C_\alpha + L_\alpha) + C_R]\norm{\alpha^{n} - \opt{\alpha}},
        \\
        \label{ineq:S_u_exactness_essential}
        C_u\norm{\alpha^{n} - \opt{\alpha}}+ \pi_u\norm{\alpha^{n+1}-\alpha^n}
        \le
        \kappa_u  C_u \bigl(\norm{\alpha^n - \opt{\alpha}} - \norm{\alpha^{n+1} - \alpha^n}\bigr)
        \le
        \kappa_u  C_u \norm{\alpha^{n+1}-\opt{\alpha}}
        \shortintertext{and}
        \label{ineq:adjoint_exactness_essential}
        (C_p + C_S\inv \kappa_u C_u)\norm{\alpha^{n}-\opt{\alpha}}
        +
        (\pi_p + C_S\inv \kappa_u \pi_u)\norm{\alpha^{n+1}-\alpha^n}
        \le
        \kappa_p  C_p \norm{\alpha^{n+1}-\opt{\alpha}}
        .
    \end{gather}
\end{lemma}

\begin{proof}
    Using the $\alpha$-update of \cref{alg:abstract}, we estimate
    \[
    \begin{aligned}[t]
        \norm{\alpha^{n+1}-\alpha^n}
        &
        =
        \norm{[\prox_{\sigma R}(\alpha^n - \sigma p^{n+1}\grad_u J(u^{n+1}))- \alpha^n]-[\prox_{\sigma R}(\opt\alpha - \sigma \opt p\grad_u J(\opt u))-\opt \alpha]}
        \\
        &
        \le
        \norm{[\prox_{\sigma R}(\alpha^n - \sigma p^{n+1}\grad_u J(u^{n+1}))- \alpha^n]-[\prox_{\sigma R}(\alpha^n - \sigma \opt p\grad_u J(\opt u))-\alpha^n]}
        \\
        \MoveEqLeft[-1]
        +
        \norm{[\prox_{\sigma R}(\alpha^n - \sigma \opt p\grad_u J(\opt u))- \alpha^n]-[\prox_{\sigma R}(\opt\alpha - \sigma \opt p\grad_u J(\opt u))-\opt \alpha]}.
    \end{aligned}
    \]
    Since proximal maps are $1$-Lipschitz, and $R$ is by \cref{ass:abstract:main}\,\ref{item:abstract:main:outer-objective} locally prox-$\sigma$-contractive at $\opt\alpha$ for $\opt p\grad_u J(\opt u)$ within $B(\opt{\alpha}, r) \isect \Dom R$ with factor $C_R$, it follows
    \begin{equation}
        \label{eq:prox-est}
        \norm{\alpha^{n+1}-\alpha^n}
        \le
        \sigma \norm{p^{n+1}\grad_u J(u^{n+1})- \opt p\grad_u J(\opt u)}
        + \sigma C_R \norm{\alpha^n-\opt\alpha}
        =: \sigma A + \sigma C_R \norm{\alpha^n-\opt\alpha}.
    \end{equation}

    We have $\opt p\grad_u J(\opt u)=\grad_\alpha S_u(\opt\alpha) \grad_u J(S_u(\opt\alpha))=\grad_\alpha(J \circ S_u)(\opt\alpha)$, where $\grad_{\alpha}(J\circ S_u)$ is $L_\alpha$-Lipschitz in $B(\opt{\alpha}, r)\ni \alpha^n$ by \cref{ass:abstract:main}\,\ref{item:abstract:main:outer-objective}. Hence
    \[
    \begin{aligned}[t]
        A
        &
        \le
        \norm{p^{n+1}\grad_u J(u^{n+1}) - \grad_{\alpha}(J\circ S_u)(\alpha^n) + \grad_{\alpha}(J\circ S_u)(\alpha^n) -\opt{p}\,\grad_u J(\opt{u})}
        \\
        &
        \le
        \norm{p^{n+1}\grad_u J(u^{n+1}) - \grad_{\alpha}S_u(\alpha^n)\grad_u J(S_u(\alpha^n))} + L_\alpha\norm{\alpha^n - \opt{\alpha}}.
    \end{aligned}
    \]
    Using the Lipschitz continuity of $\grad_u J$ from \cref{ass:abstract:main}\,\ref{item:abstract:main:outer-objective}, we continue
    \[
    \begin{aligned}[t]
        A
        &
        \le
        \norm{p^{n+1}(\grad_u J(u^{n+1})-\grad_u J(S_u(\alpha^n)) + (p^{n+1}-\grad_{\alpha}S_u(\alpha^n))\grad_u J(S_u(\alpha^n)) } + L_\alpha \norm{\alpha^n - \opt{\alpha}}
        \\
        &
        \le
        \norm{p^{n+1}}_{\inv Q}L_{\grad J}\norm{u^{n+1}-S_u(\alpha^n)}_{Q}
        + \norm{p^{n+1}-\grad_{\alpha}S_u(\alpha^n)}_{\inv Q}\norm{\grad_u J(S_u(\alpha^n))}_{Q}
        + L_\alpha \norm{\alpha^n - \opt{\alpha}}.
    \end{aligned}
    \]
    Since $\alpha^n\in B(\opt\alpha, r),$ we have  $\norm{p^{n+1}}_{\inv Q}\leq N_{p}$ by \cref{lemma:p-np} and $\norm{\grad_u J(S_u(\alpha^n))}_{Q} \leq N_{\grad J}$ by the definition in \cref{ass:abstract:main}\,\cref{item:abstract:main:testing-params}.
    Using \cref{grad:closeness_formulas} therefore gives
    \[
    A
    \le
    N_p L_{\grad J} C_u \norm{\alpha^n - \opt{\alpha}} + N_{\grad J} C_p \norm{\alpha^n - \opt{\alpha}} + L_\alpha \norm{\alpha^n - \opt{\alpha}}
    = (C_{\alpha} + L_\alpha)\norm{\alpha^n - \opt{\alpha}}.
    \]
    Inserting this into \eqref{eq:prox-est}, we obtain \eqref{ineq:alpha_in_exact_grad}.
    \Cref{ass:abstract:main}\,\cref{item:abstract:main:outer-step-length} and \eqref{ineq:alpha_in_exact_grad} then yield
    \[
    (\pi_u +\kappa_u C_u)\norm{\alpha^{n+1} - \alpha^n}
    \le
    \sigma (\pi_u +\kappa_u C_u)[C_{\alpha} + L_\alpha + C_R]
    \norm{\alpha^n - \opt{\alpha}}
    \le
    (\kappa_u -1)C_u \norm{\alpha^n - \opt{\alpha}}.
    \]
    Rearranging terms and finishing with the triangle inequality we get \eqref{ineq:S_u_exactness_essential}. We obtain \eqref{ineq:adjoint_exactness_essential} analogously to \eqref{ineq:S_u_exactness_essential}.
\end{proof}

We discussed taking gradient steps instead of proximal steps with respect to $R$ in \cite[Remark 3.10]{suonpera2022bilevel}.
We next prove that if both inner problem and adjoint iterates have small error, and we take a short step in the outer problem, then also the next inner problem and adjoint iterate has small error.

\begin{lemma}\label{lemma:grad:u_u_alpha_closeness_formula}
    Let $k \in \N$.
    Suppose \cref{ass:abstract:main} holds.
    If $\alpha^n\in B(\opt\alpha, r)$ and \cref{grad:closeness_formulas} holds for $n=k,$
    then
    \cref{grad:closeness_formulas} holds for $n=k+1$ and we have $\alpha^{k+1}\in B(\opt\alpha, 2r)$.
\end{lemma}

\begin{proof}

    First, we show that $\alpha^{k+1}\in B(\opt\alpha, 2r)$ and $u^{k+1} \in B_Q(\opt u, r_u)$.
    We use \eqref{ineq:S_u_exactness_essential} of \cref{lemma:alpha-option}. Its first inequality readily implies either $\norm{\alpha^{k} - \opt{\alpha}} > \norm{\alpha^{k+1} - \alpha^{k}}$ or $\alpha^{k+1} = \opt\alpha\in B(\opt\alpha, 2r)$.
    In the former, using $\alpha^k\in B(\opt{\alpha},r)$, we get
    \[
    \norm{\alpha^{k+1} - \opt{\alpha}} \leq \norm{\alpha^{k+1} - \alpha^{k}} + \norm{\alpha^{k}- \opt{\alpha}} < 2 \norm{\alpha^{k}- \opt{\alpha}} \leq 2r.
    \]
    We estimate using \cref{grad:u_u_alpha_closeness_formula} for $n=k$ and Lipschitz continuity of $S_u$
    \begin{equation*}
    	\begin{aligned}
    		\norm{u^{k+1} - \opt u}_Q
    		&
    		\le
    		\norm{u^{k+1} - S_u(\alpha^k)}_Q
    		+
    		\norm{S_u(\alpha^k) - S_u(\opt\alpha)}_Q
    		\\
    		&
    		\le
    		(C_u + L_{S_u})\norm{\alpha^k - \opt\alpha}
    		\le (C_u + L_{S_u})r \le r_u.
    	\end{aligned}
    \end{equation*}
    Thus $u^{k+1} \in B_Q(\opt u, r_u)$, so that we may apply \Cref{ass:abstract:main}\,\cref{item:abstract:main:inner-tracking,item:abstract:main:adjoint:tracking} with $u=u^{k+1}$.

    The former with \cref{grad:u_u_alpha_closeness_formula} for $n=k$ then imply
    \begin{equation}
    \label{ineq:inner-tracking-use}
    \begin{aligned}[t]
        \kappa_u\norm{u^{k+2}-S_u(\alpha^{k+1})}_{Q}
        &
        \leq
        \norm{u^{k+1}-S_u(\alpha^{k})}_{Q} + \pi_u\norm{\alpha^{k+1}-\alpha^{k}}
        \\
        &
        \leq
        C_u \norm{\alpha^{k} - \opt{\alpha}} + \pi_u\norm{\alpha^{k+1}-\alpha^{k}}.
    \end{aligned}
    \end{equation}
    Inserting \eqref{ineq:S_u_exactness_essential} here, we establish \cref{grad:u_u_alpha_closeness_formula} for $n=k+1$.

    Similarly, \cref{ass:abstract:main}\,\ref{item:abstract:main:adjoint:tracking}, \cref{ineq:inner-tracking-use} and \cref{grad:adjoint_closeness_formula} for $n=k$
    imply
    \[
    \begin{aligned}[t]
        \kappa_p\norm{p^{k+2}-\grad_{\alpha}S_u(\alpha^{k+1})}_{\inv Q}
        &
        \leq
        \norm{p^{k+1}-\grad_{\alpha} S_u(\alpha^{k})}_{\inv Q}
        +
        C_S\norm{u^{k+2} - S_u(\alpha^{k+1})}_{Q}
        +
        \pi_p\norm{\alpha^{k+1} -  \alpha^{k}}
        \\
        &
        \leq
        (C_p + C_S\inv\kappa_u C_u) \norm{\alpha^{k} -  \opt{\alpha}}
        +
        (\pi_p + C_S\inv\kappa_u \pi_u)\norm{\alpha^{k+1} -  \alpha^{k}}.
    \end{aligned}
    \]
    Inserting \eqref{ineq:adjoint_exactness_essential} here, we establish \cref{grad:adjoint_closeness_formula} for $n=k+1.$
\end{proof}

The next lemma is a crucial monotonicity-type estimate for the outer problem.
It depends on an $\alpha$-relative exactness condition on the inner and adjoint variables.

\begin{lemma}
    \label{lemma:outer-problem-monotonicity}
    Let $n \in \N$.
    Suppose \cref{ass:abstract:main}\,\cref{item:abstract:main:outer-objective,item:abstract:main:testing-params} hold with $\alpha^n \in B(\opt\alpha, r)$ and \cref{grad:closeness_formulas}.
    Then, for any $t>0$ and $d > 0$, for some $q^{n+1} \in \subdiff R(\alpha^{n+1})$,
    \begin{multline}
        \label{alpha_row_general}
        \iprod{p^{n+1}\grad_u J(u^{n+1}) + q^{n+1}}{\alpha^{n+1}-\opt \alpha} \geq - \frac{L_\alpha}{4t} \norm{\alpha^{n+1}- \alpha^n}^2
        \\
        + \left(\frac{\gamma_\alpha-tL_\alpha}{2}-\frac{C_\alpha}{2d}\right)\norm{\alpha^{n+1}-\opt \alpha}^2
        + \left(\frac{\gamma_\alpha-tL_\alpha}{2}-\frac{C_\alpha d}{2}\right)\norm{\alpha^n-\opt \alpha}^2.
    \end{multline}
\end{lemma}

\begin{proof}
    The $\alpha$-update of \cref{alg:abstract} in implicit form reads
    \begin{equation}
        \label{eq:alpha-update:implicit}
        0 = \sigma(q^{n+1} + p^{n+1}\grad_u J(u^{n+1}))  + \alpha^{n+1} - \alpha^n
        \quad\text{for some}\quad
        q^{n+1} \in \subdiff R(\alpha^{n+1}).
    \end{equation}
    Similarly, $0 \in H(\opt u, \opt p, \opt \alpha)$ implies
    $
    \opt p\,\grad_{u}J(\opt u) +  \opt q =0
    $
    for some
    $
    \opt q \in \subdiff R(\opt \alpha).
    $
    Writing $E_0$ for the left hand side of \eqref{alpha_row_general},
    these expressions and the monotonicity of $\partial R$ yield
    \begin{equation}\label{grad_alpha_iprod}
        \begin{aligned}[t]
            E_0
            &
            =
            \iprod{p^{n+1}\grad_u J(u^{n+1})- \opt p\,\grad_{u}J(\opt u) + q^{n+1} - \opt q}{\alpha^{n+1}-\opt \alpha}
            \\
            &
            \ge
            \iprod{p^{n+1}\grad_u J(u^{n+1})-\grad_{\alpha}S_u(\alpha^n)\grad_u J(S_u(\alpha^n))}{\alpha^{n+1}-\opt \alpha}
            \\
            \MoveEqLeft[-1]
            + \iprod{\grad_{\alpha}S_u(\alpha^n)\grad_u J(S_u(\alpha^n))-\opt p\,\grad_{u}J(\opt u)}{\alpha^{n+1}-\opt \alpha}
            =: E_1 + E_2.
        \end{aligned}
    \end{equation}

    We estimate $E_1$ and $E_2$ separately.  \cref{ass:abstract:main}\,\ref{item:abstract:main:outer-objective} and \cref{thm:stongly-convex-monotonicity} yield
    \begin{equation}
        \label{3-point_Hessian_2}
        \begin{aligned}[t]
            E_2
            &
            \ge
            \frac{\gamma_\alpha - tL_\alpha}{2}(\norm{\alpha^{n+1}-\opt \alpha}^2 + \norm{\alpha^n-\opt \alpha}^2) - \frac{L_\alpha}{4t}\norm{\alpha^{n+1}- \alpha^n}^2
        \end{aligned}
    \end{equation}
    for $t>0.$ To estimate $E_1$ we rearrange
    \[
    \begin{aligned}[t]
        E_1
        &
        =
        \iprod{p^{n+1}\grad_u J(u^{n+1})-\grad_{\alpha}S_u(\alpha^n)\grad_u J(S_u(\alpha^n))}{\alpha^{n+1}-\opt \alpha}
        \\
        &
        =
        \iprod{p^{n+1}(\grad_u J(u^{n+1})-\grad_u J(S_u(\alpha^n))) + (p^{n+1}-\grad_{\alpha}S_u(\alpha^n))(\grad_u J(S_u(\alpha^n))}{\alpha^{n+1}-\opt{\alpha}}.
    \end{aligned}
    \]
    We have $\norm{p^{n+1}}_{\inv Q} \leq N_p$ by \cref{lemma:p-np} and $\norm{\grad_u J(S_u(\alpha^n))}_{Q} \leq N_{\grad J}$ by the definition of the latter in
    \cref{ass:abstract:main}\,\cref{item:abstract:main:testing-params} with $\alpha^n\in B(\opt{\alpha}, r).$ The same assumptions establish that $\grad_u J$ is Lipschitz. Hence it holds%
    \[
    \begin{aligned}[t]
        E_1
        &
        \geq
        - \norm{p^{n+1}}_{\inv Q}\norm{\grad_u J(u^{n+1})-\grad_u J(S_u(\alpha^n))}_{Q}
        \norm{\alpha^{n+1}-\opt \alpha}
        \\
        \MoveEqLeft[-1]
        - \norm{p^{n+1}-\grad_{\alpha}S_u(\alpha^n)}_{\inv Q}\norm{\grad_u J(S_u(\alpha^n))}_{Q} \norm{\alpha^{n+1}-\opt \alpha}
        \\
        &
        \geq
        -\left(L_{\grad J}N_p
        \norm{u^{n+1}-S_u(\alpha^n)}_{Q}
        + N_{\grad J}\norm{p^{n+1}-\grad_{\alpha}S_u(\alpha^n)}_{\inv Q}
        \right) \norm{\alpha^{n+1}-\opt \alpha}.
    \end{aligned}
    \]
    Applying \eqref{grad:closeness_formulas} and Young's inequality now yields for any $d>0$ the estimate
    \begin{equation}
        \label{grad_J_estimate_u}
        \begin{aligned}[t]
            E_1
            &
            \geq
            -\left(L_{\grad J}N_p C_u + C_p N_{\grad J}\right) \norm{\alpha^n-\opt \alpha}\norm{\alpha^{n+1}-\opt \alpha}
            \\
            &
            \geq
            -\left(L_{\grad J}N_p C_u + C_p N_{\grad J}\right) \left(\frac{d}{2}\norm{\alpha^n-\opt \alpha}^2 + \frac{1}{2d}\norm{\alpha^{n+1}-\opt \alpha}^2\right).
        \end{aligned}
    \end{equation}
    By inserting \eqref{3-point_Hessian_2} and \eqref{grad_J_estimate_u} into \eqref{grad_alpha_iprod} we obtain the claim \eqref{alpha_row_general}.
\end{proof}

By combining the previous lemmas, we prove an estimate from which local convergence is immediate.

\begin{lemma}
	\label{Fejer_lemma_up_implication}
    Suppose \cref{ass:abstract:main}  and the inner and adjoint exactness estimate \eqref{grad:closeness_formulas} hold with $\alpha^n\in B(\opt\alpha,r)$ for a $n\in\N.$ Then
    \begin{equation}
        \label{Fejer_monot_up}
        (1 + \sigma\epsilon_\sigma)\norm{\alpha^{n+1} - \opt{\alpha}}^2
        \leq
        \norm{\alpha^{n} - \opt{\alpha}}^2
    \end{equation}
    for
    \[
    \epsilon_\sigma \defeq
    \gamma_\alpha - \sigma L_\alpha^2/2  -  C_\alpha^2 (\gamma_\alpha - \sigma L_\alpha^2/2)^{-1}
    >
    0.
    \]
\end{lemma}

\begin{proof}
    We observe that $\epsilon_\sigma>0$ by the inequality in \cref{ass:abstract:main}\,\ref{item:abstract:main:testing-params}.
    Since $\alpha^n\in B(\opt\alpha,r)$ and \eqref{grad:closeness_formulas} hold, \cref{lemma:outer-problem-monotonicity} with $t=\sigma L_\alpha/2$ gives the monotonicity estimate
    \begin{multline*}
        \iprod{
            \sigma \left(
            p^{n+1}\grad_u J(u^{n+1}) + q^{n+1}
            \right)
        }{
            \alpha^{n+1}-\opt{\alpha}
        }
        \ge
        \\
        - \frac{1}{2} \norm{\alpha^{n+1} - \alpha^{n}}^2
        + \frac{\sigma(\gamma_\alpha - \sigma L_\alpha^2/2 - C_{\alpha}\inv d)}{2} \norm{\alpha^{n+1} - \opt \alpha}^2
        +\frac{\sigma(\gamma_\alpha - \sigma L_\alpha^2/2 - C_{\alpha} d)}{2}\norm{\alpha^{n}-\opt \alpha}^2
    \end{multline*}
    for some
    $q^{n+1} \in \subdiff R(\alpha^{n+1})$ and a constant $d>0.$
    Choosing $d = \inv C_{\alpha}(\gamma_{\alpha} - \sigma L_\alpha^2/2),$ we get
    \[
    \gamma_{\alpha} - \sigma L_\alpha^2/2 - C_{\alpha}\inv d = \epsilon_\sigma
    \quad\text{and}\quad
    \gamma_{\alpha} - \sigma L_\alpha^2/2 - C_{\alpha}d = 0.
    \]
	It follows
    \begin{equation}
        \label{ineq:grad-ineq-to-prove}
        \iprod{
            \sigma \left(
            p^{n+1}\grad_u J(u^{n+1}) + q^{n+1}
            \right)
        }{
            \alpha^{n+1}-\opt{\alpha}
        }
        \ge
        - \frac{1}{2} \norm{\alpha^{n+1}- \alpha^{n}}^2
        + \frac{\sigma\epsilon_\sigma}{2} \norm{\alpha^{n+1}-\opt \alpha}^2.
    \end{equation}

    We now come to the fundamental argument of the testing approach of \cite{tuomov-proxtest}, combining operator-relative monotonicity estimates with the three-point identity.
    Indeed, \eqref{ineq:grad-ineq-to-prove} combined with the implicit algorithm
    \[
    	0 = \sigma(q^{n+1} + p^{n+1}\grad_u J(u^{n+1}))  + \alpha^{n+1} - \alpha^n
    	\quad\text{for some}\quad
    	q^{n+1} \in \subdiff R(\alpha^{n+1})
    \]
     gives
    \[
        \iprod{\alpha^{n+1} - \alpha^n}{\alpha^{n+1} - \opt\alpha}
        +
        \frac{\sigma \epsilon_\sigma}{2}\norm{\alpha^{n+1} - \opt\alpha}^2
        \le
        \frac{1}{2}\norm{\alpha^{n +1}-\alpha^{n}}^2.
    \]
    Inserting the three-point identity \eqref{3-point-identity} yields \eqref{Fejer_monot_up}.
\end{proof}

We simplify the assumptions of the previous lemma to just \cref{ass:abstract:main}.

\begin{lemma}\label{Fejer_lemma_up}
    Suppose \cref{ass:abstract:main} holds. Then
    \eqref{Fejer_monot_up} holds for any $n\in\N$.
\end{lemma}

\begin{proof}
    The claim readily follows from \cref{Fejer_lemma_up_implication} if
    we prove by	induction for all $n\in\N$ that
    \begin{equation}
        \label{eq:GIFB:fejer:induction}
        \alpha^n\in B(\opt\alpha, r),\, \eqref{grad:closeness_formulas}, \text{ and }  \eqref{Fejer_monot_up}
        \text{ hold}.
    \end{equation}
    We first prove \eqref{eq:GIFB:fejer:induction} for $n=0.$
    \Cref{ass:abstract:main}\,\cref{item:abstract:main:testing-params} directly establishes \eqref{grad:closeness_formulas}, and \cref{ass:abstract:main}\,\cref{item:abstract:main:local}
    establishes $\alpha^n\in B(\opt \alpha, r).$
    Now \cref{Fejer_lemma_up_implication} proves \eqref{Fejer_monot_up} for $n=0.$ This concludes the proof of induction base.

    We then make the induction assumption that \eqref{eq:GIFB:fejer:induction} holds for $n=k$
    and prove it for $n=k+1.$
    The induction assumption and \cref{lemma:grad:u_u_alpha_closeness_formula} give \eqref{grad:closeness_formulas} for $n=k+1.$  The inequality  \eqref{Fejer_monot_up} for	$n=k$ and $\alpha^k\in B(\opt \alpha, r)$  also ensure $\norm{\alpha^{k+1}-\opt \alpha} \le \norm{\alpha^{k}-\opt \alpha} \leq r.$
    Now \cref{Fejer_lemma_up_implication} shows \eqref{Fejer_monot_up} and concludes the proof of \cref{eq:GIFB:fejer:induction} for $n=k+1$.
\end{proof}

We finally come to the main convergence result for the abstract \cref{alg:abstract}.

\begin{theorem}
    \label{thr:grad-grad-grad_convergence}
    Suppose \cref{ass:abstract:main} holds. Then for $\{\alpha^n\}_{n \in \N}$ generated by \cref{alg:abstract}, we have $\norm{\alpha^{n} - \opt \alpha}^2 \to 0$ linearly.
\end{theorem}

\begin{proof}
    \Cref{Fejer_lemma_up} proves \eqref{Fejer_monot_up} for all $n\in\N.$
    Since, by the very same lemma, $1+\sigma\epsilon_\sigma > 1$ therein, the claim follows.
\end{proof}

\begin{corollary}
     Suppose \cref{ass:abstract:main} holds. Then for $\{u^n\}_{n \in \N}$ generated by \cref{alg:abstract}, we have $\norm{u^n - \opt  u}^2 \to 0$ linearly.
\end{corollary}

\begin{proof}
    We apply \cref{thr:grad-grad-grad_convergence} and then use \cref{grad:u_u_alpha_closeness_formula} and Lipschitz continuity of $S_u$ in
    \[
        \norm{u^{n+1}-\opt{u}}_{Q}
        \leq \norm{u^{n+1}-S_u(\alpha^{n})}_{Q}  + \norm{S_u(\alpha^{n}) - S_u(\opt{\alpha})}_{Q}
        \leq (C_u+ L_{S_u})\norm{\alpha^{n} - \opt{\alpha} }.
        \qedhere
    \]
\end{proof}

\section{Splitting methods for the inner problem}
\label{sec:inner}

In \cref{sec:inner-general} we present a general framework for optimisation algorithms that can be used for solving the inner problem. We also give examples of algorithms following the framework. In \cref{sec:inner-tracking} we first prove inner tracking estimate for such general algorithm under appropriate assumptions. Last, we show that the aforementioned example algorithms satisfy these assumptions and thus also the inner tracking property.

\subsection{A general approach}
\label{sec:inner-general}

Let $U$ and $\AlphaSpace$ be Hilbert spaces, and $G=T+W$ for some $T, W: U \times \AlphaSpace \to U$.
Then $u=S_u(\alpha)$ solves the inner problem if
\begin{equation}
	\label{AplusB-opt-cond}
	0= T(u; \alpha) + W(u;\alpha).
\end{equation}
We assume that such a solution exists for all $\alpha\in\AlphaSpace_{2r},$ see \cref{rem:GIFB:solution-map-existence}.\footnote{We do not here at the outset assume unique solutions, although we will later, in practise, need to impose uniqueness. In principle, for \cref{sec:abstract}, sufficient regularity of some solution map $S_u$ and a corresponding tracking inequality is enough.}
We seek to move towards $S_u(\alpha)$ with preconditioned forward-backward type algorithms that, for fixed $\alpha$, given the previous iterate $u^k$, solve the next iterate $u^{k+1}$ from
\begin{equation}
	\label{genaral-precontitioned-alg}
	0 = T(u^{k+1}; \alpha) + W(u^k;\alpha) + Q(u^{k+1} - u^k)
\end{equation}
for some self-adjoint preconditioning operator $Q \in \linear(U; U)$.

\begin{example}[Forward-backward splitting]
	\label{ex:forward-backward}
	Let $U$ and $\AlphaSpace$ be Hilbert spaces. Let further $f, g: U \times \AlphaSpace \to \extR$ be convex, proper and lower semicontinuous in their first parameter, $f$ also differentiable with respect to its first parameter.
	Then \cref{AplusB-opt-cond} with $T=\grad_u f$ and $W=\partial_u g$ is a necessary and sufficient optimality condition for solution of $\min_{u\in U} f(u;\alpha) + g(u;\alpha)$.
	Choosing $Q = \inv\tau\Id,$ the implicit form \cref{genaral-precontitioned-alg} corresponds to the forward-backward splitting
	\begin{equation}
		\label{alg:forward-backward}
		u^{k+1} = \prox_{\tau g(\freevar; \alpha)}(u^k - \tau \grad f(u^k;\alpha))
	\end{equation}
\end{example}

\begin{example}[Primal-dual proximal splitting, PDPS]
    \label{ex:pdps}
    Let $X, Y$ and $\AlphaSpace$ be Hilbert spaces, $f_0, e: X \times \AlphaSpace \to \extR$ and $g: Y \times \AlphaSpace \to \extR$ be convex, proper, and lower semicontinuous in their first parameter, $e$ with an $L$-Lipschitz gradient with respect to its first parameter.
    Define $f \defeq f_0 + e$, suppose $K\in\linear(X;Y)$, and consider the problem
    \[
        \min_x f(x, \alpha) + g(Kx; \alpha).
    \]
    This can be equivalently written as the saddle point problem
    \[
        \min_x \max_y f(x; \alpha) + \iprod{Kx}{y} - g^*(y; \alpha).
    \]
    Following \cite[Theorem 5.11]{he2012convergence,clason2020introduction} and writing $u=(x,y)$, this problem has primal-dual optimality conditions of the form \cref{AplusB-opt-cond} with
    \begin{equation}
        \label{PDPS-A-and-B}
        T(u; \alpha)
        :=
        \begin{pmatrix}
            \partial_x f_0(x;\alpha) + K^*y\\
            \inv\omega[\partial_y g^*(y;\alpha) - Kx]
        \end{pmatrix}
        \quad
        \text{and}
        \quad
        W(u; \alpha)
        :=
        \begin{pmatrix}
            \grad_x e(x;\alpha)\\
            0
        \end{pmatrix}
    \end{equation}
    for any $\omega>0$.
    Defining for step length parameters $\tau_x,\tau_y>0$ the preconditioning operator
    \begin{equation}
        \label{PDPS-M}
        Q :=
        \begin{pmatrix}
            \inv\tau_x\Id & - K^* \\
            - K & \inv\omega\inv\tau_y\Id
        \end{pmatrix},
    \end{equation}
    \cref{genaral-precontitioned-alg} then becomes an implicit form of the primal-dual proximal splitting (PDPS) method of \cite{chambolle2010first}, here with an additional forward step with respect to $e$:
    \begin{equation}
        \label{alg:PDPS}
        \begin{cases}
            x^{k+1} &= \prox_{\tau_x f_0(\freevar; \alpha)}(x^k - \tau_x(K^*y^k + \grad_x e(x^k;\alpha)),  \\
            y^{k+1} &= \prox_{\tau_y g^*(\freevar; \alpha)} (y^k + \tau_y K((1+\omega)x^{k+1} - \omega x^k)).
        \end{cases}
    \end{equation}
    The method converges (for fixed $\alpha$) if $\tau_x L/2 + \tau_x\tau_y\norm{K}^2 \le 1$ and $\omega=1$, or is chosen to appropriately reflect available strong convexity; see, e.g., \cite{clason2020introduction}.
\end{example}

Further algorithms fitting the general framework \cref{genaral-precontitioned-alg} can be found in \cite{clason2020introduction,tuomov-proxtest}.

\subsection{Tracking estimates}
\label{sec:inner-tracking}

We now prove inner problem tracking estimates for the above methods.
To do so, we start with a result for abstract methods of the form \cref{genaral-precontitioned-alg}.
It requires introducing abstract forms of monotonicity.
Let $X$ be a Hilbert space and $\tilde X \subset X.$
We say that $T:X \to X$ is $(\Gamma_T$\emph{-strongly}) \emph{monotone} for some self-adjoint and positive semi-definite $\Gamma_T \in \linear(X;X)$ at $\widehat x\in \tilde X$ if
\begin{equation}
	\label{ineq:A-monotonicity}
    \iprod{T(x) - T(\widehat x)}{x - \widehat x} \ge \norm{x - \widehat x }_{\Gamma_T}^2 \quad (x\in \tilde X).
\end{equation}
If this holds for all $\widehat{x}\in \tilde X,$ we say that $T$ is ($\Gamma_T$-strongly) monotone in $\tilde X.$
Moreover, we say that $W: X \to X$ is ($\Gamma_W$\emph{-strongly}) \emph{three-point monotone} at $\widehat x\in \tilde X$ with respect to a self-adjoint and positive semi-definite $\Gamma_W, \Lambda \in \linear(X; X)$ if
\begin{equation}
	\label{ineq:B-three-point-monotonicity}
    \iprod{W(z) - W(\widehat x)}{x - \widehat x} \ge
    \norm{x - \widehat x }_{\Gamma_W}^2
    -\frac{1}{4}\norm{x - z }_{\Lambda}^2 \quad (x,z\in \tilde X).
\end{equation}
If this holds for all $\widehat{x}\in \tilde X,$ we say that $W$ is ($\Gamma_W$-strongly) three-point monotone with respect to $\Lambda$ in $\tilde X.$
When $\Gamma_T = \gamma \Id, \Gamma_P = 0$ and $\Lambda = L\Id$, these operator-relative properties reduce to standard strong and three-point monotonicity, see \cite[Cororolary 7.2]{clason2020introduction}.

We now prove the promised tracking estimate result for general splitting methods for the inner problem.
The Lipschitz properties required from $S_u$, were discussed in \cref{rem:GIFB:solution-map-existence}.

\begin{theorem}
	\label{thm:generel-algorithm-inner-tracking}
	Let $U$ and $\AlphaSpace$ be Hilbert spaces,
	$u^k \in \tilde U\subset U, \alpha,\breve\alpha\in \tilde{\AlphaSpace}\subset \AlphaSpace,$ and operators $T,W : U \to U.$
    Let further $Q, \Gamma := \Gamma_T + \Gamma_W, \Lambda \in \linear (U;U)$ be self-adjoint and positive semi-definite, and suppose that $T$ is ( $\Gamma_T$-strongly) monotone in $\tilde U$ and $W$ is ( $\Gamma_W$-strongly)  three-point monotone with respect to $\Lambda$ in $\tilde U.$
	Moreover, assume that $u=S_u(\alpha)$ that satisfies \cref{AplusB-opt-cond} is unique, and $L_{S_u}$-Lipschitz in $\tilde{\AlphaSpace}.$
	If
	\begin{subequations}
		\label{eq:conds}
		\begin{align}
			\label{eq:cond1}
			Q + 2\Gamma & \ge c Q \quad \text{for some} \quad  c>1,
            \quad\text{and}
            \\
			\label{eq:cond2}
			Q & \ge \Lambda/2,
		\end{align}
	\end{subequations}
	then $u^{k+1}$ generated by \cref{genaral-precontitioned-alg} satisfies \cref{ass:abstract:main}\,\cref{item:abstract:main:inner-tracking} with $\kappa_u=\sqrt{c}$ and $\pi_u=L_{S_u}$, i.e.
	\begin{equation*}
		\sqrt{c}\norm{u^{k+1} - S_u(\alpha)}_{Q}
		\leq
		\norm{u^k - S_u(\breve\alpha)}_{Q}
		+
		L_{S_u} \norm{\alpha - \breve\alpha}.
	\end{equation*}
\end{theorem}

\begin{proof}
	By the assumed monotonicity properties of $T$ and $W,$ we have
	\begin{equation}
		\label{ineq:inner-monotonicity}
		\iprod{T(u^{k+1};\alpha) + W(u^k;\alpha)}{u^{k+1} - S_u(\alpha)}
		\ge
		\norm{u^{k+1} - S_u(\alpha)}^2_{\Gamma} - \frac{1}{4}\norm{u^{k+1} - u^k}^2_{\Lambda}.
	\end{equation}
	Inserting  \cref{genaral-precontitioned-alg}  into \cref{ineq:inner-monotonicity}, and applying the three-point inequality \cref{3-point-identity} yields
	\[
		\frac{1}{2}\norm{u^{k+1} - S_u(\alpha)}^2_{Q + 2\Gamma} + \frac{1}{2}\norm{u^{k+1} - u^k}^2_{Q-\Lambda/2}
		\leq
		\frac{1}{2}\norm{u^k - S_u(\alpha)}^2_{Q}.
	\]
    Using \cref{eq:conds} we obtain
	\begin{equation*}
		\kappa_u\norm{u^{k+1} - S_u(\alpha)}_{Q} \le \norm{u^k - S_u(\alpha)}_{Q}.
	\end{equation*}
	Therefore triangle inequality and Lipschitz continuity of $S_u$ in $\tilde{\AlphaSpace}$ gives
	\[
		\kappa_u\norm{u^{k+1} - S_u(\alpha)}_{Q}
		\le \norm{u^k - S_u(\breve\alpha)}_{Q} + L_{S_u}\norm{\alpha - \breve\alpha}.
        \qedhere
	\]
\end{proof}

The next two theorems specialise \cref{thm:generel-algorithm-inner-tracking} to forward-backward splitting and the PDPS.

\begin{theorem}
	\label{thm:inner-tracking-FB}
	Let $U$ and $\AlphaSpace$ be Hilbert spaces with
	$\tilde U\subset U$ and $\breve\alpha\in\tilde{\AlphaSpace}\subset \AlphaSpace.$
	Let further $f:U\times \AlphaSpace\to\R$ and $g:U\times \AlphaSpace\to\overline\R$ be convex, proper, and lower semicontinuous in their first argument.
    Suppose that $f$ is differentiable, $\grad f$ is $L$-Lipschitz continuous, and $g$ is $\gamma$-strongly convex in $\tilde U$ for some $\gamma>0$.
    Moreover, assume that $S_u$ satisfying \cref{AplusB-opt-cond} with $T=\grad_u f$ and $W=\partial_u g$ is single-valued and $L_{S_u}$-Lipschitz in $\tilde{\AlphaSpace}.$ If the step length parameter $\tau > 0$ satisfies $\tau L\le 2,$ then $u^{k+1}$ generated by
	forward-backward algorithm \cref{alg:forward-backward} for given $u^k \in \tilde U$ and $\alpha\in \tilde{\AlphaSpace}$ satisfies inner tracking property.
\end{theorem}
\begin{proof}
	This proof is based on \cref{thm:generel-algorithm-inner-tracking} and to use it we need to prove \cref{eq:conds}.
	By assumption, $\partial g$ is $\gamma$-strongly monotone, and $\grad f$ is three-point monotone with the factor $L.$ Thus the condition \cref{eq:conds} now reads $1 + 2\gamma > 1, 1 \ge \tau L/2$ and $1>0$. These hold due to $\gamma>0$ and the step length condition $\tau L\le 2$.
\end{proof}

\begin{remark}
	Similar proof works for strongly convex $f$ and convex $g$. Gradient descent is a special case of this with $g$ as the zero function.
\end{remark}

\begin{theorem}
	\label{thm:inner-tracking-PDPS}
	Let $U=(X\times Y)$ and $\AlphaSpace$ be Hilbert spaces with
	$\tilde U = (\tilde X\times \tilde Y)\subset U$ and $\breve\alpha\in\tilde{\AlphaSpace}\subset \AlphaSpace.$
	Let further $f_0, e: X \times \AlphaSpace \to \extR$ and $g: Y \times \AlphaSpace \to \extR$ be convex, proper, and lower semicontinuous in their first argument, $e$ with a $L-$Lipschitz gradient with respect to its first argument.
	Moreover, suppose $f_0$ is $\gamma_f$-strongly convex in $\tilde X$ and $g^*$ is $\gamma_{g^*}$-strongly convex in $\tilde Y$ for some $\gamma_f,\gamma_{g^*}>0$.
	Also let $K\in\linear(X;Y),$ and
	assume that $S_u$ satisfying \cref{AplusB-opt-cond} with \cref{PDPS-A-and-B} is single-valued and $L_{S_u}$-Lipschitz in $\tilde{\AlphaSpace}.$
	If the step length parameters $\tau_x>0$ and $\tau_y>0$ satisfy $\tau_x L/2 + \tau_x\tau_y\norm{K}^2 \le 1$, then $u^{k+1} = (x^{k+1}, y^{k+1})$ generated by the PDPS \cref{alg:PDPS}  for given $u^k = (x^k, y^k) \in \tilde U$ and $\alpha\in \tilde{\AlphaSpace}$
	satisfies the inner tracking property.
\end{theorem}

\begin{proof}
	The PDPS can be written as \cref{genaral-precontitioned-alg} with operators $T$ and $W$ as in \cref{PDPS-A-and-B}, and preconditioning operator $Q$ defined in \cref{PDPS-M}.
	The operator $T$ is $\Gamma$-strongly monotone in $\tilde U$ and $W$ is three-point monotone with respect to $\Lambda$ in $\tilde U$ for
	\[
        \Gamma
        :=
        \begin{pmatrix}
            \gamma_f\Id & (1-\inv\omega)K^*/2 \\
            (1-\inv\omega)K/2 & \inv\omega \gamma_{g^*}\Id
        \end{pmatrix}
        \quad
        \text{and}
        \quad
        \Lambda
        =
        \begin{pmatrix}
            L\Id & 0 \\
            0 & 0
        \end{pmatrix}.
	\]

    The claim follows from \cref{thm:generel-algorithm-inner-tracking} if we show \cref{eq:conds}.
    We first expand \cref{eq:cond1} as
	\begin{equation*}
		\begin{pmatrix}
			((1-c)\inv\tau_x + 2\gamma_f)\Id & [(1-\inv\omega) - (1 - c)]K^* \\
			[(1-\inv\omega) - (1 - c)]K & \inv\omega((1-c)\inv\tau_y + 2\gamma_{g^*})\Id
		\end{pmatrix}
		\ge 0.
	\end{equation*}
    This readily holds by taking
    $
        \inv\omega = c = 1+2\min\{\tau_x\gamma_f,\tau_y\gamma_{g^*} \}.
    $
	To prove \cref{eq:cond2}, we use Young's inequality to estimate for any $\delta > 0$ that
	\[
		2\iprod{x}{K^*y}
		\le
		\inv\tau_x(1-\delta)\norm{x}^2 + \tau_x\inv{(1-\delta)}\norm{K^*y}^2
        \quad\text{for all}\quad x, y.
	\]
    Thus
	\begin{equation*}
		Q
		=
		\begin{pmatrix}
			\inv\tau_x\Id & -K^* \\
			-K & \inv\omega\inv\tau_y\Id
		\end{pmatrix}
		\ge
		\begin{pmatrix}
			\delta\inv\tau_x\Id & 0 \\
			0 & \inv\omega\inv\tau_y\Id - \tau_x\inv{(1-\delta)}KK^*
		\end{pmatrix}.
	\end{equation*}
	Consequently \cref{eq:cond2} holds if
	$
		\delta \ge \tau_x L/2
    $
    and
	$
		(1-\delta)\inv\omega > \tau_x\tau_y\norm{K}^2.
	$
	Since $\inv\omega>1,$ the latter holds if $1-\delta \ge \tau_x\tau_y\norm{K}^2$.
    It remains to take $\delta = \tau_x L/2$, and use our assumption $\tau_x L/2 + \tau_x\tau_y\norm{K}^2 \le 1$.
    Note that, since $\inv\omega > 1$, this also proves that $Q$ is positive definite.
\end{proof}

\section{Splitting methods for the adjoint}
\label{sec:adjoint}

We now develop splitting methods for solving the adjoint equation based on conventional iterative splitting methods for linear systems.
In \cref{sec:adjoint-general} we present the overall approach and prove a perturbed contractivity property. We also provide example of splittings that satisfy the relevant conditions.
In \cref{sec:adjoint-tracking} we then prove the adjoint tracking estimate based on the perturbed contractivity property.

\subsection{Operator splitting methods for linear systems}
\label{sec:adjoint-general}

Let $U$ and $\AlphaSpace$ be Hilbert spaces.
We need to find $\nexxt p\in \linear(U; \AlphaSpace)$ that approximately solves the adjoint equation
\begin{equation}
    \label{eq:adjoint-general:problem}
	\nexxt p \grad_u\adjointGeneral(\nexxt u,\alpha^k) + \grad_{\alpha}\adjointGeneral(\nexxt u,\alpha^k) = 0
\end{equation}
to such a precision that adjoint tracking property \cref{ass:abstract:main}\,\cref{item:abstract:main:adjoint:tracking} holds.
Dropping the iteration indices for brevity, this is a linear equation of the form
\begin{gather}
	\label{eq:linear-system}
	A_v p = b_v
\shortintertext{with}
    \label{eq:adjoint-general:A-for-problem}
    v=(u,\alpha),
    \quad
    A_v p := p \grad_u \adjointGeneral(u, \alpha),
    \quad\text{and}\quad
    b_v = -\grad_{\alpha} \adjointGeneral(u, \alpha).
\end{gather}
We write this section in general Hilbert spaces $V$ and $P$, which in the overall setting of this paper equal $U \times \AlphaSpace$ and $\linear(U; \AlphaSpace)$.
For some $v \in D\subset V$ and $A_v \in \linear(P; P)$, let us thus be given the splitting
\begin{equation}
	\label{eq:splitting}
	A_v=N_v+M_v
\end{equation}
with $N_{v}$ invertible.
Given the previous iterate $p^k\in P$, a linear system splitting method takes a step towards solution of equation \cref{eq:linear-system} by solving $\nexxt p$ from
\begin{equation}
	\label{eq:splitting:update}
	N_v p^{k+1} + M_v p^k = b_v.
\end{equation}

The next lemma is adapted from \cite{jensen2022nonsmooth}, where squared tracking estimates were needed for PDE-constrained optimisation. We need non-squared estimates.
The condition \eqref{eq:convergence:splitting:a-condition} is a Lipschitz condition on $A_v$, which for \eqref{eq:adjoint-general:problem} becomes a Lipschitz condition on $\grad_u\adjointGeneral$; see \eqref{eq:adjoint-general:A-for-problem}.
The condition \eqref{eq:convergence:splitting:split-condition} is a standard contractivity condition on the splitting scheme; see \cite[Theorem 10.1.1]{golub1996matrix}.
It obviously holds for the trivial splitting $N_v=A_v$ and $M_v=0$ if $A_v$ is positive definite. For \eqref{eq:adjoint-general:problem} this again means that $\grad_u\adjointGeneral$ is positive definite.
We provides examples of other splitting schemes after the lemma.

\begin{lemma}
	\label{lemma:splitting:helper}
	Let $V$ and $P$ be Hilbert spaces, and $A_v \in \linear(P; P)$ for all $v \in D \subset V$.
    Suppose $v \mapsto b_v: V \to P$ is $L_b$-Lipschitz, and for some $L_A \ge 0$ that
	\begin{equation}
		\label{eq:convergence:splitting:a-condition}
		\norm{A_v - A_{\tilde v}} \le L_A \norm{v-\tilde v}
		\quad (v, \tilde v \in D).
	\end{equation}
	Let us be given a splitting \cref{eq:splitting}
	with $N_v$ invertible, and such that
	\begin{equation}
		\label{eq:convergence:splitting:split-condition}
		\norm{\inv N_v M_v} \le \zeta
		\quad\text{and}\quad
		\gamma_N\norm{\inv N_v} \le 1
	\end{equation}
    for some $\zeta \in [0, 1)$ and $\gamma_N>0$.
	Suppose $\breve p$ solves $N_v \breve p + M_v p = b_v$ for a given $p\in P.$
	Then for any $v, \bar v \in D$ and $\bar p \in P$ satisfying
	\[
    	A_{\bar v} \bar p =  b_{\bar v},
	\]
	we have
	\begin{equation}
		\label{eq:splitting:helper:target}
		\zeta \norm{p-\bar p}
		\ge
		\norm{\breve p-\bar p}
		- (L_A\norm{\bar p} + L_b)\inv\gamma_N\norm{v - \bar v}.
	\end{equation}
\end{lemma}

\begin{proof}
	Using \eqref{eq:splitting:update} with $A_{\bar v} \bar p =  b_{\bar v}$ and $A_{v} \bar p = N_{v}\bar p + M_{v}\bar p$, we expand
	\[
	\begin{aligned}[t]
		\breve p - \bar p
		&
		= \inv N_{v}(b_v - M_{v} p) - \bar p
		\\
		&
		= \inv N_{v}[b_v - b_{\bar v}] + \inv N_{v}(A_{\bar v} - A_{v})\bar p -  \inv N_{v}M_{v} (p - \bar p).
	\end{aligned}
	\]
	Taking the norm and using the triangle inequality, the Lipschitz assumption on $b$, as well as \eqref{eq:convergence:splitting:a-condition}, therefore
	\[
	\norm{\breve p - \bar p}
	\le
	\norm{\inv N_{v}}(L_A\norm{\bar p} + L_b)\norm{\bar v-v}
	+ \norm{\inv N_{v}M_{v}}\norm{p - \bar p}.
	\]
	Further using \eqref{eq:convergence:splitting:split-condition}, we obtain \eqref{eq:splitting:helper:target}.
\end{proof}

If the parameters are not perturbed, we immediately get linear convergence:

\begin{corollary}
    If $v=\bar v$ in \cref{lemma:splitting:helper}, then
    $
		\norm{\breve p-\bar p}
		\le
		\zeta \norm{p-\bar p}.
	$
\end{corollary}

We next present example slitting methods, and conditions that guarantee \eqref{eq:convergence:splitting:split-condition}.

\begin{example}[Jacobi splitting]
	\label{ex:splitting:jacobi}
	If $N_v$ is the diagonal of a matrix $A_v \in \R^{n \times n}$, we obtain Jacobi splitting.
	The first part of \eqref{eq:convergence:splitting:split-condition} reduces to strict diagonal dominance, see \cite[§10.1]{golub1996matrix}. The second part always holds and $N_v$ is invertible when the main diagonal of $A_v$ has only positive entries. Then $\gamma_N$ is the minimum of the diagonal values.
\end{example}

\begin{example}[Gauss--Seidel splitting]
	\label{ex:splitting:Gauss–Seidel}
	If $N_v$ is the lower triangle and diagonal of a matrix $A_v \in \R^{n \times n}$, we obtain Gauss--Seidel splitting.
	The first inequality of \eqref{eq:convergence:splitting:split-condition} holds for some $\zeta \in [0, 1)$ when $A_v$ is symmetric and positive definite. This follows similarly to \cite[proof of Theorem 10.1.2]{golub1996matrix}.
	The second inequality holds for some $\gamma_N$ when $N_v$ is invertible, i.e., has no zeros on the main diagonal.
\end{example}

\begin{example}[Identity splitting]
	\label{ex:splitting:id-splitting}
	If $\gamma\Id \le A_v$ and $\norm{A_v}\le L$ for $\gamma, L > 0,$ and $N_v=\inv\theta \Id$ for  $0 < \theta < \min\{\inv\gamma_N, 2\gamma/L^2\},$ then \eqref{eq:convergence:splitting:split-condition} holds for $\zeta = \sqrt{1 + \theta^2L^2 - 2\theta\gamma}\in [0, 1)$ and (any) $\gamma_N > 0.$
    Indeed,
     \[
    \norm{\inv N_v M_v}^2
    =
    \sup_{x, \norm{x}\le 1}\iprod{(\Id - \theta A_v)x}{(\Id - \theta A_v)x}
    \le
    1 + \theta^2\norm{A}^2 - 2\theta \inf_{x, \norm{x}\le 1} \iprod{Ax}{x}
    \le
    1 + \theta^2L^2 - 2\theta\gamma.
    \]
\end{example}

\begin{example}[No splitting]
	\label{ex:splitting:no-splitting}
	If $N_v=A_v$, \eqref{eq:convergence:splitting:split-condition} holds with $\zeta = 0$ and $\gamma_N$ the minimal eigenvalue of $A_v$, assumed symmetric positive definite.
\end{example}

\begin{example}[Block Gauss–Seidel]
	\label{ex:splitting:block-GS}
	Let
	\[
	A_v
	=
	\begin{pmatrix}
		A_{11} & A_{12} \\
		A_{21} & A_{22}
	\end{pmatrix},
	\,
	N_v
	=
	\begin{pmatrix}
		N_{11} & 0 \\
		A_{21} & N_{22}
	\end{pmatrix},
	\,
	\text{ and }
	\,
	M_v
	=
	\begin{pmatrix}
		M_{11} & A_{12} \\
		0 & M_{22}
	\end{pmatrix}
	\]
	with $A_{11}=N_{11} + M_{11}$ and $A_{22}=N_{22} + M_{22}$ for some invertible and bounded operators $N_{11}$ and $N_{22}$ invertible and bounded.
    Suppose that $\gamma_N$ and $\zeta \in [0, 1)$ satisfy
	\begin{equation}
	    \label{eq:block-gs:gamma}
        0
        <
        \gamma_N
        \leq
        \frac{\norm{N_{11}}\norm{N_{22}}}{2\norm{N_{11}} + \norm{N_{22}}(1 + \norm{\inv N_{22}A_{21}}^2)},
	\end{equation}
    and
	\begin{multline*}
        \norm{\inv N_{22}(M_{22}-A_{21}\inv N_{11}A_{12})}^2
        +
        \norm{\inv N_{22}(M_{22}-A_{21}\inv N_{11}A_{12})}\norm{\inv N_{22} A_{21}}\norm{\inv N_{11} M_{11}}
        \\
        +
        \norm{\inv N_{11} M_{11}}^2(1+\norm{\inv N_{11} A_{12}} )(1 + \norm{\inv N_{22} A_{21}}^2)
        +
        \norm{\inv N_{11} A_{12}}^2
        +
        \norm{\inv N_{11} A_{12}}
        \norm{\inv N_{11} M_{11}}
        \le
        \zeta^2.
	\end{multline*}
    Then \eqref{eq:convergence:splitting:split-condition} holds, as we prove in \cref{sec:block-gs}.
    In particular, if $M_{11}=0$ and $M_{22}=0$, the condition on $\zeta$ reduces to the much simpler
	\begin{equation}
	    \label{eq:block-GS:cond}
        \norm{\inv{A_{22}}A_{21}\inv{A}_{11}A_{12}}^2
        +
        \norm{\inv{A}_{11} A_{12}}^2
        \le
        \zeta^2.
	\end{equation}
	This holds, in particular, when the diagonal blocks of $A_v$ are invertible and large compared to the off-diagonal blocks.
	Moreover, there exists  $\gamma_N$ satisfying \eqref{eq:block-gs:gamma} when both $N_{11}$ and $N_{22}$ are non-zero.
\end{example}

\begin{example}
	\label{ex:splitting:block-GS-spesific}
    Let $G=T+W$ as in \cref{ex:pdps} on the PDPS for the inner problem.
    Assume that the functions $f$ and $g^*$ therein are twice continuously differentiable in their first parameter and $\grad_x f$ and $\grad_y g^*$ are also Lipschitz differentiable in their second parameter.
    We may then apply the block Gauss–Seidel splitting of \cref{ex:splitting:block-GS} to the corresponding adjoint equation, which now reads
	\[
        \begin{pmatrix}
            \grad_x^2f(x;\alpha)  &  K^* \\
            -K & \grad_y^2 g^*(y;\alpha)
        \end{pmatrix}
        \begin{pmatrix}
            p_x \\
            p_y
        \end{pmatrix}
        +
        \begin{pmatrix}
            \grad_{\alpha}\grad_x f(x;\alpha) \\
            \grad_{\alpha}\grad_y g^*(y;\alpha)
        \end{pmatrix}
        = 0
	\]
	for the unknown $p=(p_x, p_y)\in P.$
	The simplified condition \eqref{eq:block-GS:cond} holds if $\grad_x^2 f(x;\alpha)$ and  $\grad_y^2 g^*(y;\alpha)$ are non-singular, and $K$ is small compared to them.
	In practise, to ensure invertibility and the stability of inversion we can take $N_{11}$ as perturbed version of $\grad_x^2f(x;\alpha).$
	We also take $N_{22} = \grad_y^2 g^*(y;\alpha) + \inv\theta_y\Id$ for $\theta_y>0,$ which imply the choice of $M_v$ via $M_v = A_v - N_v.$
\end{example}

\subsection{Tracking estimates}
\label{sec:adjoint-tracking}

We next show adjoint equation tracking estimates for the above operator splitting methods for linear systems as $\alpha$ varies. Following result relies heavily on \cref{lemma:splitting:helper}.

\begin{theorem}
	Let $U$ and $\AlphaSpace$ be Hilbert spaces.
	Let further $\adjointGeneral: U \times \AlphaSpace \to U$ and $S_u: \AlphaSpace \to U$ be Lipschitz continuously differentiable, and satisfy
	\[
		\grad_{\alpha}S_u(\alpha)\grad_u G(S_u(\alpha);\alpha) + \grad_\alpha G(S_u(\alpha); \alpha)=0.
	\]
	Also let $Q \in \linear(U; U)$ be positive definite and self-adjoint.
    For $v = (u, \alpha)\in U\times \AlphaSpace$, define $A_v \in \linear(P; P)$ for $P=\linear(U; \AlphaSpace)$ and $b_v$ via
    \[
         A_v p := p \grad_u G(u; \alpha)
         \quad\text{and}\quad
         b_v = -\grad_{\alpha} G(u; \alpha).
    \]
	Let the splitting $N_v + M_v = A_v$ satisfy the conditions of \cref{lemma:splitting:helper}
	for some $\zeta \in [0, 1)$ and $\gamma_N>0$.
	Then $\nexxt p$ solving \cref{eq:splitting:update} for given $p^k\in P$ satisfies the adjoint tracking property, i.e.
	\begin{equation*}
		\kappa_p\norm{\nexxt p - \grad_{\alpha}S_u(\alpha)}_{\inv Q}
		\leq
		\norm{p^k - \grad_{\alpha}S_u(\breve{\alpha})}_{\inv Q}
		+
		C_S\norm{u - S_u(\alpha)}_Q
		+
		\pi_p \norm{\alpha - \breve{\alpha}}
	\end{equation*}
	for any $\breve{\alpha}\in V_{\alpha}\subset \AlphaSpace$ as well as the constants
    \[
        \kappa_p = \inv\zeta,
        \quad
        C_S = \sup_{\alpha \in V_{\alpha}} \inv\zeta(L_A\norm{ \grad_{\alpha}S_u(\alpha)}_{\inv Q}  + L_b),
        \quad\text{and}\quad
        \pi_p = L_{\grad_{\alpha}S_u}.
    \]
\end{theorem}

\begin{proof}
	By \cref{lemma:splitting:helper}, we have
	\begin{equation*}
		\begin{aligned}
			 \zeta \norm{p^k-\grad_{\alpha}S_u(\alpha)}_{\inv Q}
			\ge
			\norm{\nexxt p-\grad_{\alpha}S_u(\alpha)}_{\inv Q}
			- (L_A\norm{ \grad_{\alpha}S_u(\alpha)}_{\inv Q} + L_b)\inv\gamma_N\norm{u - S_u(\alpha)}_Q.
		\end{aligned}
	\end{equation*}
	Rearranging terms and multiplying by $\inv\zeta$ gives
	\begin{equation*}
		\begin{aligned}
			\kappa_p\norm{\nexxt p-\grad_{\alpha}S_u(\alpha)}_{\inv Q}
			\le
			\norm{p^k-\grad_{\alpha}S_u(\alpha)}_{\inv Q}
			+ C_S\norm{u - S_u(\alpha)}_Q.
		\end{aligned}
	\end{equation*}
	The triangle inequality and Lipschitz continuity of $\grad_{\alpha} S_u$ gives
	\begin{equation*}
		\begin{aligned}
			\norm{p^k-\grad_{\alpha}S_u(\alpha)}_{\inv Q}
			&
			\le
			\norm{p^k-\grad_{\alpha}S_u(\breve{\alpha})}_{\inv Q}
			+
			\norm{\grad_{\alpha}S_u(\breve{\alpha})-\grad_{\alpha}S_u(\alpha)}_{\inv Q}
			\\
			&
			\le
			\norm{p^k -\grad_{\alpha}S_u(\breve{\alpha})}_{\inv Q}
			+
			L_{\grad_\alpha S_u}\norm{\breve{\alpha} - \alpha}.
		\end{aligned}
	\end{equation*}
	Combining the previous estimates finishes the proof.
\end{proof}

We needed to assume that the solution mapping of the inner problem is Lipschitz continuously differentiable to prove the adjoint tracking estimates. We discuss next the reasonability of this assumption. If the inner and adjoint solution mappings $S_u$ and $S_p$ are Lipschitz, then also $\grad_{\alpha} S_u$ is since
$
	\grad_{\alpha} S_u(\alpha) = S_p(S_u(\alpha), \alpha).
$
The Lipschitz properties of $S_u$ are discussed in \cref{rem:GIFB:solution-map-existence}.
We next prove the Lipschitz continuity of $S_p$ under similar assumption as in \cref{lemma:splitting:helper}.

\begin{lemma}\label{lemma:S_p-Lipschitz}
	Let $U$ and $\AlphaSpace$ be Hilbert spaces.
	Let further $Q \in \linear(U; U)$ be positive definite, and $\grad\adjointGeneral: U \times \AlphaSpace \to \linear(U;U \times \AlphaSpace)$ be invertible and Lipschitz continuous with constant $L_{\grad\adjointGeneral}$ in some bounded closed set $V_u \times V_{\alpha}.$
	Moreover, assume that $\norm{(\grad_u \adjointGeneral(u; \alpha))^{-1}}\leq \gamma_\adjointGeneral^{-1}$
	in $V_u \times V_{\alpha}$ for some $\gamma_\adjointGeneral>0$.
	Then $S_p$ defined in \cref{def:S_p} is Lipschitz in $V_u \times V_{\alpha},$ i.e.
	\[
	\norm{S_p(u_1, \alpha_1) - S_p(u_2, \alpha_2)}_{\inv Q} \leq L_{S_p}(\norm{u_1 - u_2}_Q+\norm{\alpha_1-\alpha_2})
	\]
	for $u_1,u_2\in V_u$ and $\alpha_1,\alpha_2 \in V_{\alpha}$ with the factor
	\[
        L_{S_p} \defeq \gamma_\adjointGeneral^{-1} L_{\grad\adjointGeneral}
        \Bigl(
            1 +
            \gamma_\adjointGeneral^{-1} \max_{(u,\alpha)\in V_u\times V_{\alpha}}\norm{\grad_{\alpha}\adjointGeneral(u, \alpha)}_{\inv Q}
        \Bigr).
    \]
\end{lemma}
\begin{proof}
	Using the definition of $S_p$ in \eqref{def:S_p}, we rearrange
	\[
	\begin{aligned}[t]
		S_p(u_1, \alpha_1) - S_p(u_2, \alpha_2)
		&
		=
		\left( \grad_{\alpha} \adjointGeneral (u_1; \alpha_1) - \grad_{\alpha} \adjointGeneral (u_2; \alpha_2)\right)(\grad_u \adjointGeneral(u_1; \alpha_1))^{-1}
		\\
		\MoveEqLeft[-1]
		+ \grad_{\alpha} \adjointGeneral (u_2; \alpha_2)\left( (\grad_u \adjointGeneral(u_1; \alpha_1))^{-1} - 	(\grad_u \adjointGeneral(u_2; \alpha_2))^{-1}\right) .
	\end{aligned}
	\]
	Since $Q$ is positive definite, hence invertible, we have $\norm{pu}_\AlphaSpace = \norm{(pQ^{-1/2})(Q^{1/2}u)} \le \norm{p}_{Q^{-1}}\norm{u}_Q$. This and the triangle inequality give
	\begin{equation}
		\label{S_p-ineq1}
		\begin{aligned}[t]
			\MoveEqLeft[1]\norm{S_p(u_1, \alpha_1) - S_p(u_2, \alpha_2)}_{\inv Q}
			\\
			&
			\leq
			\norm{(\grad_u \adjointGeneral(u_1; \alpha_1))^{-1}}\norm{ \grad_{\alpha} \adjointGeneral (u_1; \alpha_1) - \grad_{\alpha} \adjointGeneral (u_2; \alpha_2)}_{\inv Q}
			\\
			\MoveEqLeft[-1]
			+ \norm{ \grad_{\alpha} \adjointGeneral (u_2; \alpha_2)}_{\inv Q}\norm{(\grad_u \adjointGeneral(u_1; \alpha_1))^{-1} - (\grad_u G(u_2; \alpha_2))^{-1}}
			=: E_1 + E_2.
		\end{aligned}
	\end{equation}
	The bound $\norm{(\grad_u \adjointGeneral(u; \alpha))^{-1}}\leq \gamma_\adjointGeneral^{-1}$ and the Lipschitz continuity of $\grad \adjointGeneral$ in $V_u\times V_{\alpha}$ give
	\begin{equation}
		\label{S_p-ineq2}
		E_1 \leq \gamma_\adjointGeneral^{-1} L_{\grad_{\alpha} \adjointGeneral}\left(\norm{u_1-u_2}_Q + \norm{\alpha_1 - \alpha_2}\right).
	\end{equation}
	Towards estimating the second term on the right hand side of \eqref{S_p-ineq1}, we observe that
	\[
	A^{-1} - B^{-1}= A^{-1}B B^{-1} - A^{-1}A B^{-1} = A^{-1}(A-B)B^{-1}
	\]
	for any invertible linear operators $A, B$.
	This and Lipschitz continuity of $\grad \adjointGeneral$ give
	\[%
	\begin{aligned}[t]
		E_2
		&
		=
		\norm{\grad_u \adjointGeneral(u_1;\alpha_1)^{-1}(\grad_u \adjointGeneral(u_1;\alpha_1)- \grad_u \adjointGeneral(u_2;\alpha_2))\grad_u \adjointGeneral(u_2;\alpha_2)^{-1}}
		\cdot \norm{ \grad_{\alpha} \adjointGeneral (u_2; \alpha_2)}_{\inv Q}
		\\
		&
		\leq
		\left(
		\max_{(u,\alpha)\in V_u\times V_{\alpha}}\norm{\grad_{\alpha}\adjointGeneral(u, \alpha)}_{\inv Q}
		\right)
		\norm{\grad_u \adjointGeneral(u_1;\alpha_1)^{-1}}
        \\
        \MoveEqLeft[-5]
        \cdot
        \norm{\grad_u \adjointGeneral(u_2;\alpha_2)^{-1}}\norm{(\grad_u \adjointGeneral(u_1;\alpha_1)- \grad_u \adjointGeneral(u_2;\alpha_2))}
		\\
		&
		\leq
		\gamma_\adjointGeneral^{-2} L_{\grad \adjointGeneral}
		\left(
		\max_{(u,\alpha)\in V_u\times V_{\alpha}}\norm{\grad_{\alpha}\adjointGeneral(u, \alpha)}_{\inv Q}
		\right)
		\left(\norm{u_1-u_2}_Q + \norm{\alpha_1 - \alpha_2}\right).
	\end{aligned}
	\]
	Inserting this inequality and \eqref{S_p-ineq2} into \eqref{S_p-ineq1} establishes the claim since $\norm{\grad_{\alpha}\adjointGeneral(u, \alpha)}_{\inv Q}$ is bounded in $ V_u\times V_{\alpha}$ by the Lipschitz continuity of $\grad \adjointGeneral$.
\end{proof}

\section{Numerical experiments}
\label{sec:numerical}

In this section, we evaluate the operator identification performance of the following methods:
\begin{description}
	\item[PDPS + block-GS]
    \Cref{alg:PDPS-block}, i.e., our general bilevel approach with the PDPS of \cref{ex:pdps} for the inner problem, and the block Gauss–Seidel splitting of \cref{ex:splitting:block-GS-spesific} for the adjoint.

	\item[PDPS + identity]
    Our general bilevel approach with PDPS for the inner problem, and the identity splitting of \cref{ex:splitting:id-splitting} for the adjoint.

	\item[implicit]
    Solve both the inner problem and the adjoint equation to a high precision.

    \item[trust region] The inexact derivative-free method of \cite{ehrhardt2021inexact} that solves inner problem inexactly while taking trust region steps to solve the bilevel (outer) problem.
\end{description}
In all cases except the trust region method, the outer iterates are updated with forward-backward steps.

	Note that the algorithms PDPS + block-GS and PDPS + identity follow the framework of \cref{alg:abstract}, i.e. satisfy the inner and adjoint tracking conditions, if the objective functions satisfy the assumptions of \cref{thm:inner-tracking-PDPS} and the conditions discussed in the splitting \cref{ex:splitting:block-GS,ex:splitting:id-splitting} respectively.

The implicit method is used as a baseline comparison and the trust region method as more advanced one. We have chosen the latter as a comparison from the more recent literature, as the algorithm itself requires the least information as a derivative-free method, and, therefore, can, in principle, be applied to the (primal-only form of) the same inner problem as our method, unlike many other methods that would require a different smoothing of the inner regularisation term $g$ to achieve the required second-order differentiability, as we will discuss below. The specific convergence conditions of \cite{ehrhardt2021inexact} can be difficult to verify, however, and might indirectly require the second-order differentiability of the inner regulariser.

Our specific operator identification problems are:
\begin{description}
	\item[MRI] Find an optimal sparse subsampling (line) pattern for MRI reconstruction using four slices of the digital brain phantom of \cite{belzunce_2018_1190598} as a ground-truth in the outer problem, and corresponding simulated fully sampled MRI measurements as the data of the corresponding inner-problem(s).
	\item[deblurring] Find an optimal parametrisation of convolution kernel and regularisation parameter to deblur/deconvolve a photograhic image with simulated blur.
\end{description}

We define component functions used in both problems in \cref{sec:objective}.
Then in \cref{sec:mri,sec:deblur} we present inner problem and adjoint equations details for MRI and deblurring respectively.
We present numerical details in \cref{sec:numerical-setup}, and discuss performance in \cref{sec:results}.

\begin{algorithm}
	\caption{PDPS + block-GS}
	\label{alg:PDPS-block}
	\begin{algorithmic}[1]
		\Require
		Functions $J: U \to \R$, and $R: \AlphaSpace \to \extR$, with $J$ Fréchet differentiable and $R$ convex, on Hilbert spaces $U = X\times Y$ and $\AlphaSpace$.
		Functions $f_0, e: X \times \AlphaSpace \to \extR$ and $g: Y \times \AlphaSpace \to \extR$ convex in their first parameter, $e$ has a $L-$Lipschitz gradient with respect to its first parameter.
		Also let $K\in\linear(X;Y),$
		Outer step length $\sigma>0$, adjoint step length $\theta_y > 0,$ inner step lengths $\tau_x, \tau_y>0$ satisfying $\tau_x L/2 + \tau_x\tau_y\norm{K}^2 \le 1$ and parameter $\omega>0.$

		\State Pick an initial iterate $(u^{0}, p^{0}, \alpha^{0}) \in U\times \linear(U; \AlphaSpace) \times \AlphaSpace.$
		\For{$k \in \N$}
		\State		$x^{k+1} \defeq \prox_{\tau_x f_0(\freevar; \alpha^k)}(x^k - \tau_x(K^*y^k + \grad_x e(x^k;\alpha^k))$
		\Comment{inner primal update}
		\State		$y^{k+1} \defeq \prox_{\tau_y g^*(\freevar; \alpha^k)} (y^k + \tau_y K((1+\omega)x^{k+1} - \omega x^k))$
		\Comment{inner dual update}
		\State Choose a splitting $\grad_x^2 f(\nexxt x; \alpha^k) = N_{11} + M_{11}$ with $N_{11}$ invertible.
		\State Write $N_{22} = \grad_y^2 g^*(\nexxt y; \alpha^k) + \inv\theta_y\Id$
		\State $ \nexxt p_x \defeq
		\inv N_{11} \left(
		-M_{11}p_x^k	- \grad_{\alpha}\grad_x f(\nexxt x; \alpha^k) - K^*p_y^k
		\right)$
		\Comment{adjoint primal update}
		\State $\nexxt p_y \defeq
		\inv N_{22}\left(
		\inv\theta_yp_y^k - \grad_{\alpha}\grad_y g^*(\nexxt y; \alpha^k) - Kp_x^{k+1}
		\right)$
		\Comment{adjoint dual update}
		\State $\alpha^{k+1} \defeq \prox_{ \sigma R}\left (\alpha^{k} - \sigma  p^{k+1}\grad_{u}J(u^{k+1})\right )$
		\Comment{outer forward-backward step}
		\EndFor
	\end{algorithmic}
\end{algorithm}

\subsection{Objective functions}
\label{sec:objective}

For a training set of $m$ “ground truth” images $\{b_i\}$ of dimension
$n_1 \times n_2$, i.e. each $b_i \in \R^{n_1n_2}_+,$ we take as the outer data fitting term
\[
    J(u) = \frac{1}{2}\sum_{i=1}^m\norm{x_i-b_i}_2^2, \qquad u_i = (x_i, y_i) \in \R^{n_1n_2}_+\times \R^{2n_1n_2}.
\]
We specify the outer regulariser $R$ and the (primal-dual) inner problem for the different experiments in corresponding subsections.

For simplicity, we have not yet analysed our method in a stochastic setting, although a generalisation should be immediate.\footnote{This would simply involve selecting on each iteration $k$ a random subset $S_k \in \{1,\ldots,m\}$, and making the step with respect to $\frac{m}{2\#S_k}\sum_{i \in S_k}\norm{x_i-b_i}$ instead of $J$. Then steps for the inner problems would only have to be made for $i \in S_k$. However, as the inner and adjoint variables for $i \not\in S_k$ would not be updated, convergence would need to be analysed carefully.}
We therefore concentrate on small sample sets, with the goal of evaluating the performance of our splitting approaches. These performance improvements should directly generalise to an eventual stochastic extension of our approach.

The general form of the inner problem, related primal-dual optimality conditions and description of the algorithm can be found in \cref{ex:pdps}.
We take $g(y; \alpha_0) = \alpha_0 \norm{y}_{2,1}$, where $\norm{y}_{2,1}$ is the sum over pixelwise two-norms.
The use the steps of the PDPS for the inner problem, we need its Fenchel conjugate
\[
	g^*(y; \alpha_0) = \sum_{j=1}^{n_1n_2} \delta_{B_{\R^2}(0,\alpha_0)}(y_j).
\]
However, our theory does not  yet allow for the inner problem to be nonsmooth.\footnote{\cref{ass:abstract:main} makes no such restriction, and we could, in principle, take $S_u$ as a differentiable selection of a multi-valued solution map. However, to prove the assumption in \cref{sec:inner,sec:adjoint}, we needed to impose differentiability and strong convexity assumptions.}
Therefore, instead of $g^*(y; \alpha_0)$, we use the $C^2$ strongly convex approximation
\begin{equation}
	\label{def:inner-reg}
	g_{\varepsilon, \delta}^*(y; \alpha_0) = \sum_{j=1}^{n_1n_2} \max\{0,\frac{1}{3\varepsilon}(\norm{y_j} - \alpha_0)^3\} + \frac{\delta}{2}\norm{y_j}^2
	\qquad (y_j \in \R^2)
\end{equation}
Since $g_{\varepsilon,\delta}^*$ is convex, proper, and lower semicontinuous, it is the Fenchel conjugate of $g_{\varepsilon,\delta} := g_{\varepsilon,\delta}^{**}.$
We take $\varepsilon=1\cdot 10^{-6}$ and $\delta=1\cdot 10^{-4},$ which are numerically practical small values that ensure that $g_{\varepsilon,\delta}^* \approx g^*$.
We differentiate and derive the proximal operator of $g_{\varepsilon, \delta}^*$ in \cref{sec:inner-prox-operator}

\begin{figure}[t]
	\begin{subfigure}[t]{0.24\textwidth}%
		\centering
		\includegraphics[width=\textwidth]{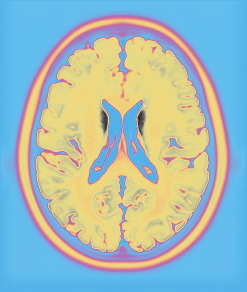}
	\end{subfigure}%
	\hfill%
	\begin{subfigure}[t]{0.24\textwidth}%
		\centering
		\includegraphics[width=\textwidth]{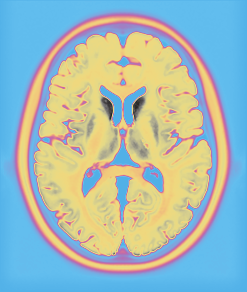}
	\end{subfigure}%
	\hfill%
	\begin{subfigure}[t]{0.24\textwidth}
		\centering
		\includegraphics[width=\textwidth]{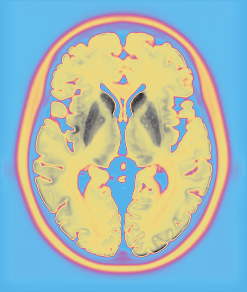}
	\end{subfigure}%
	\hfill%
	\begin{subfigure}[t]{0.24\textwidth}
		\centering
		\includegraphics[width=\textwidth]{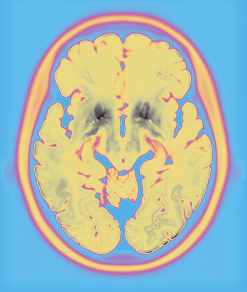}
	\end{subfigure}%
	\hfill%
	\begin{subfigure}[t]{0.24\textwidth}%
		\centering
		\includegraphics[width=\textwidth]{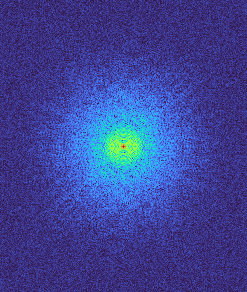}
	\end{subfigure}%
	\hfill%
	\begin{subfigure}[t]{0.24\textwidth}%
		\centering
		\includegraphics[width=\textwidth]{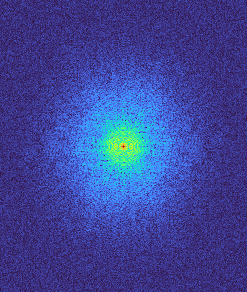}
	\end{subfigure}%
	\hfill%
	\begin{subfigure}[t]{0.24\textwidth}
		\centering
		\includegraphics[width=\textwidth]{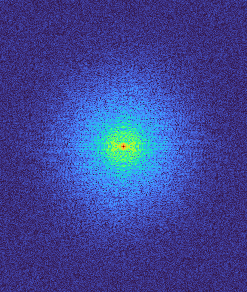}
	\end{subfigure}%
	\hfill%
	\begin{subfigure}[t]{0.24\textwidth}
		\centering
		\includegraphics[width=\textwidth]{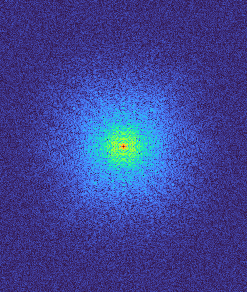}
	\end{subfigure}%
	\caption{Top: true data (4 instances) presented in the colormap \includegraphics[scale=0.5]{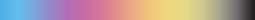}
		with values from interval $[0,1]$. Bottom: corresponding simulated fully sampled MRI measurements presented in logarithmic scale, computed as $\log_{10}(|\freevar| + 0.05),$ using colormap \includegraphics[scale=0.5]{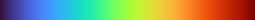} with values from 	interval $[-1.3, 2.0].$
	}
	\label{fig:MRI}
\end{figure}

\subsection{MRI}
\label{sec:mri}

The ground truth of the training set, $\{b_i\},$ contains four $247\times 292$ slices of the digital brain phantom \cite{belzunce_2018_1190598} scaled into interval $[0,1]$ and corresponding simulated fully sampled MRI measurement $\{z_i\},$ see \cref{fig:MRI}. The measurements $\{z_i\}$ are simulated by adding Gaussian noise of standard deviation $0.02$ to $\{b_i\}$ and then taking the Fourier transform.

The inner (inverse) problem is
\begin{equation}
	\label{eq:numerical:mri:inner}
	\argmin_{x_i\in \R^{n_1n_2}_+} \frac{1}{2}\norm{Z_\alpha( \mathcal F x_i-z_i)}_2^2 + g_{\varepsilon,\delta}(Dx_i ; \alpha_0)
	\qquad
	(\alpha \in [0, \infty)^{75})
\end{equation}
where $Z_\alpha$ is the subsampling operator (diagonal matrix), see \cref{fig:mask-structure}, $\mathcal{F}$ is discrete 2-D Fourier transform (matrix) and $D$ is a backward difference operator (matrix) with Dirichlet boundary conditions.
We take $\alpha_0=0.02$ constant because the subsampling operator $Z_\alpha$ can also model the regularisation factor, and we do not want several competing parameters to optimise.
This value of $\alpha_0$ was chosen by trial and error to achieve good reconstruction with fully sampled data, i.e. $Z_\alpha=\Id$. To apply \cref{ex:pdps} for the derivation of the PDPS, we split $f=f_0+e$ for $f_0(x; \alpha) = \frac{1}{2}\norm{Z_\alpha( \mathcal F x-z)}_2^2$ and $e(x; \alpha) = 0$. Although $f_0$ involves the linear forward operator $Z_\alpha( \mathcal F$, its proximal mapping can be easily calculated due to the unitarity of $\mathcal F$ and the diagonality of $Z_\alpha$.

\begin{remark}
    \label{rem:numerical:mri-f0}
    This choice of $f_0$ is not necessarily strongly convex as required for \cref{thm:inner-tracking-PDPS} to ensure the inner tracking property. It is, however, convex, and due to our good numerical results, we have not added artificial strong convexity.
    Moreover, in many cases, appropriate local growth can be elicited through the metric subregularity of the entire objective \cite{tuomov-regtheory,tuomov-subreg,clason2020introduction}.
\end{remark}

We use block Gauss–Seidel splitting for the adjoint equation following \cref{ex:splitting:block-GS-spesific}.
We have $\grad_x^2 f(u; \alpha) = \mathcal F^* Z_\alpha^2 \mathcal F$ and we take $N_{11} = \mathcal F^* Z_\theta^2 \mathcal F$ with $Z_\theta^2 = \max \{\inv\theta_x, Z_\alpha^2\}$ for
\begin{equation}
	\label{def:inv-theta-x}
 	\inv\theta_x(i, j) = 0.1 + 0.4\Bigl(1 - \sin\bigl(\tfrac{i}{n_1}\pi\bigr)\sin\bigl(\tfrac{j}{n_2}\pi\bigr)
 	\Bigr)^2, \quad (i\in\{1, ..., n_1\},\, j\in\{1, ..., n_2\})
\end{equation}
and the maximum is taken pointwise. To be more precise $\inv\theta_x$ is the matrix \cref{def:inv-theta-x} stacked into the diagonal of $n_1n_2 \times n_1n_2$ matrix.

For the outer regulariser we take
\[
    R(\alpha) = \beta\phi(w^T\alpha) + \delta_{[0, \infty)^n}(\alpha)
    \quad\text{where}\quad
    \phi(t) = t + \delta_{(-\infty, M]} (t)
\]
with the parameter $\beta=10$ and sparsity control $M=0.15$.
The vector $w$, whose components sum to $1,$ weights the components of $\alpha$ according to how many lines of the $k$-space they correspond to, see \cref{fig:mask-structure}. Our choice for the initial subsampling weights $\alpha^0$ has each component $0.15$ according to our choice for parameter $M$ such that $\alpha^0\in \Dom R$ and $\alpha^0$ doesn't contain prior information about important weights.
The proximal operator of $R$ we present in \cref{sec:outer-prox-operators}.

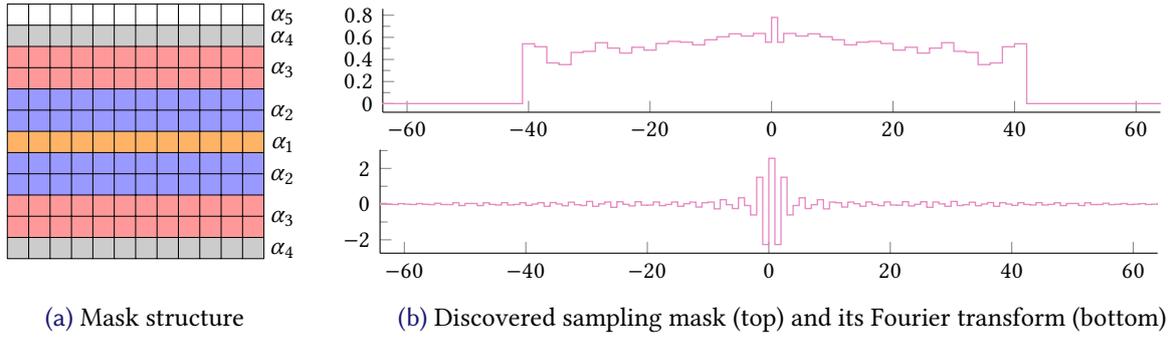
\begin{figure}
	\begin{subfigure}[b]{0.25\textwidth}
	\resizebox{1.0\linewidth}{!}{%
		\begin{tikzpicture}[inner sep=0pt,outer sep=0pt, scale=0.8]%
			\fill[gray!40!white] (0,0) rectangle (4.8,0.4);
			\fill[gray!40!white] (0,4.0) rectangle (4.8,4.4);

			\fill[red!40!white] (0,0.4) rectangle (4.8,1.2);
			\fill[red!40!white] (0,3.2) rectangle (4.8,4.0);

			\fill[blue!40!white] (0,1.2) rectangle (4.8,2.0);
			\fill[blue!40!white] (0,2.4) rectangle (4.8,3.2);

			\fill[orange!60!white] (0,2.0) rectangle (4.8,2.4);

			\node at (5.15, 2.15) {$\alpha_1$};
			\node at (5.15, 2.75) {$\alpha_2$};
			\node at (5.15, 1.55) {$\alpha_2$};
			\node at (5.15, 3.55) {$\alpha_3$};
			\node at (5.15, 0.75) {$\alpha_3$};
			\node at (5.15, 4.15) {$\alpha_4$};
			\node at (5.15, 0.15) {$\alpha_4$};
			\node at (5.15, 4.55) {$\alpha_5$};

			\draw[step=0.4cm,black,very thin] (0,0) grid (4.8,4.8);

			\node[white] at (0, -0.4) {$\alpha_5$};
		\end{tikzpicture}%
	}%
		\caption{Mask structure}
		\label{fig:mask-structure}
	\end{subfigure}
	\hspace{1em}
	\begin{subfigure}[b]{0.75\textwidth}
		\begin{tikzpicture}
			\begin{axis}[%
				width=1\linewidth,
				height=0.25\linewidth,
                xmax=64,
                xmin=-64,
				xlabel near ticks,
				ylabel near ticks,
				scaled y ticks=false,
				yminorticks=true,
				minor y tick num=1,
				xminorticks=true,
				minor x tick num=0,
				axis x line*=bottom,
				axis y line*=left,
				outer sep=0pt,
				font=\footnotesize,
				]

				\addplot[color=Set2-D, line width=0.5pt, const plot] table[x = location, y = value]{data/bilevelMRIsamplingPattersparsity015.txt};
			\end{axis}
		\end{tikzpicture}
        \\
		\begin{tikzpicture}
			\begin{axis}[%
				width=1\linewidth,
				height=0.25\linewidth,
                xmax=64,
                xmin=-64,
				xlabel near ticks,
				ylabel near ticks,
				scaled y ticks=false,
				yminorticks=true,
				minor y tick num=1,
				xminorticks=true,
				minor x tick num=0,
				axis x line*=bottom,
				axis y line*=left,
				outer sep=0pt,
				font=\footnotesize,
				]

				\addplot[color=Set2-D, line width=0.5pt, const plot] table[x = location, y = value]{data/bilevelMRIFFTsamplingPattersparsity015.txt};
			\end{axis}
		\end{tikzpicture}
		\caption{Discovered sampling mask (top) and its Fourier transform (bottom)}
		\label{fig:learned-mask}
	\end{subfigure}%
	\caption{Sampling pattern for MRI.
    \Cref{fig:mask-structure} shows how $\alpha$ parametrise the sampling pattern with a small example. Different colours represent different components of $\alpha$ such that applying the mask multiplies every pixel of one colour with same multiplier $\alpha_i$. In our case image height is $292$ and $\alpha$ has $75$ elements.
    \Cref{fig:learned-mask} presents identified pattern weights and their Fourier transform, with the $x$-axis restricted to $[-64,64]$.
	}
	\label{fig:numerical:sampling-pattern}
\end{figure}

\subsection{Deblurring}
\label{sec:deblur}

We use a cropped portion ($128\times 128$) of image 02 from the free Kodak dataset \cite{franzenkodak} converted to gray values in $[0, 1]$ as our "ground truth" $\{b_i\}.$ The $2$-D convolution operator (matrix) parametrised by $\alpha_2, \alpha_3$ and $\alpha_4$ as illustrated in \cref{fig:numerical:deblurring-param} is denoted by $A_{\alpha}.$
Operator $r_{\theta}$ rotates image $\theta$ degrees, clockwise for $\theta>0$ and counterclocwise for $\theta<0.$
We form $\{z_i\}$ by computing $r_{-1}(A_{\alpha}r_1(b_i))$ for $[\alpha_2, \alpha_3, \alpha_4] = [0.15, 0.1, 0.75] $ and adding Gaussian noise of standard deviation $0.02$.

The inner (inverse) problem is
\begin{equation}
	\label{eq:numerical:deblurring:inner}
	\argmin_{x_i\in \R^{n_1n_2}_+} \frac{1}{2}\norm{A_\alpha x_i-z_i}_2^2 + g_{\varepsilon,\delta}(Dx_i ; \alpha_0)
	\qquad
	(\alpha \in [0, \infty)^4)
\end{equation}
where $A_{\alpha}$ is the aforementioned convolution operator and $D$ is a backward difference operator.
We use the scaled regularisation parameter $\alpha_0 = C\alpha_1$ for the constant $C=\tfrac{1}{10}$ to help with outer iterate convergence by ensuring the same order of magnitude for all components of $\alpha$.
We consider $f$ in \cref{ex:pdps} as sum of $f_0(x; \alpha) = 0$ and $e(x; \alpha) = \frac{1}{2}\norm{A_\alpha x_i-z_i}_2^2,$ which means that the primal step for the the inner problem does not involve a proximal operator, but only gradient step.
Again this choice of $f_0$ is not necessarily strong convex; see, however, \cref{rem:numerical:mri-f0}.

For the block Gauss–Seidel splitting we need to choose an invertible $N_{11}.$ We use the fact that $\grad_x^2 f(u; \alpha) = A_\alpha^*A_\alpha$ equals $\mathcal F^* Z_\alpha^2 \mathcal F$ with $Z_\alpha$ being the Fourier transform of the convolution kernel, or more precisely it stacked into diagonal of $n_1n_2 \times n_1n_2$ matrix. This lets us again take $N_{11} = \mathcal F^* Z_\theta^2 \mathcal F$ such that $Z_\theta^2 = \max \{\inv\theta_x, Z_\alpha^2\}$ for $\inv\theta_x$ defined in \cref{def:inv-theta-x}.

With the parameter  $\beta=10^4$, our choice for the outer regulariser is
\[
    R(\alpha)=\beta\left(\alpha_2+\alpha_3+\alpha_4-1\right)^2+\delta_{[0, \infty)}(\alpha_1).
\]
The proximal operator of $R$ we present in \cref{sec:outer-prox-operators}.

\begin{figure}[t]
	\begin{subfigure}[t]{0.24\textwidth}%
		\resizebox{\linewidth}{!}{%
			\begin{tikzpicture}[inner sep=0pt,outer sep=0pt]%
				\fill[gray!40!white] (0.6,0) rectangle (2.4,3.0);
				\fill[gray!40!white] (0,0.6) rectangle (3.0,2.4);
				\fill[blue!40!white] (0.6,1.2) rectangle (2.4,1.8);
				\fill[blue!40!white] (1.2,0.6) rectangle (1.8,2.4);
				\fill[red!40!white] (1.2,1.2) rectangle (1.8,1.8);
				\node at (1.5, 1.5) {$\alpha_2$};
				\node at (1.5, 2.1) {$\alpha_3$};
				\node at (2.1, 2.1) {$\alpha_4$};
				\node at (0.3, 0.3) {$0$};
				\node at (0.3, 2.7) {$0$};
				\node at (2.7, 0.3) {$0$};
				\node at (2.7, 2.7) {$0$};
				\draw[step=0.6cm,black,very thin] (0,0) grid (3.0,3.0);
			\end{tikzpicture}%
		}%
		\caption{Kernel structure}
		\label{fig:numerical:deblurring-param}
	\end{subfigure}%
	\hfill%
	\begin{subfigure}[t]{0.24\textwidth}%
		\centering
		\includegraphics[width=\textwidth]{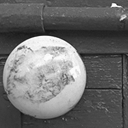}
		\caption{Original}
		\label{original-blur}
	\end{subfigure}%
	\hfill%
	\begin{subfigure}[t]{0.24\textwidth}
		\centering
		\includegraphics[width=\textwidth]{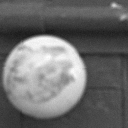}
		\caption{Blurry; error $9.8\%$}
		\label{blurred-noisy}
	\end{subfigure}%
	\hfill%
	\begin{subfigure}[t]{0.24\textwidth}
		\centering
		\includegraphics[width=\textwidth]{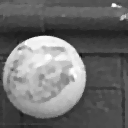}
		\caption{Result; error $6.9\%$}
		\label{FIFB-blur}
	\end{subfigure}%
	\caption{Deconvolution kernel parametrisation, data, and result for PDPS + block-GS. The different colours in \ref{fig:numerical:deblurring-param} represent different components of $\alpha$ such that elements of kernel with same colour have same value. The errors in \ref{blurred-noisy} and \ref{FIFB-blur} are $\norm{\text{image} - \text{target} }_2/\norm{\text{target}}_2$.
	}
	\label{fig:deblurring}
\end{figure}

\subsection{Numerical setup}
\label{sec:numerical-setup}

\begin{table}[t]
	\caption{
		Algorithm parametrisation, time multiplier, and total outer steps to reach threshold computational resources (CPU time) value.
		The threshold is 15000 for MRI and 10000 for deblurring, except 6000 and 3000 for PDPS + block-GS. The time multipliers allow conversion from computational resources to seconds.
		It differs between algorithms and problems due to different levels of parallelisability.
		The ‘inner steps’ and ‘adjoint steps’ indicate the (maximum) number of iterations taken towards a solution of the inner problem or the adjoint equation on every outer iteration.
		The \texttt{cgs} method (of Matlab) that we use to solve the adjoint for the implicit method in the deblurring experiment, may use a smaller number of iterations as determined by the tolerance $10^{-4}$. Similarly, the inner level solver of trust region method may take less iterations if it achieves the dynamic tolerance. Step lengths for PDPS (inner steps) are $\tau_x = 0.354,\tau_y = 0.350$ for MRI and $\tau_x = 0.600, \tau_y = 0.141$ for deblurring.
	}
	\label{tab:setup}
	\centering
	\begin{NiceTabular}{lrrrrrrr}
		&        & outer & inner & adjoint & time &        &           \\
		& method & steps & steps & steps & mult. & $\theta_x, \theta_y$ & $\sigma$  \\
		\toprule
		\Block{3-1}{\rotate MRI}
		& implicit & 16 & $3\cdot10^3$ & $200$ & $0.23$ & \cref{def:inv-theta-x}, $0.1$ & $7\cdot 10^{-4}$  \\
		& PDPS + identity & $1.9\cdot 10^3$ & $1$ & $1$ & $0.23$ & $0.1, 6.25\cdot 10^{-4}$ & $1\cdot 10^{-5}$  \\
		& PDPS + block-GS & $7.6\cdot 10^2$ & $1$ & $1$ & $0.22$ & \cref{def:inv-theta-x}, $0.1$ & $1\cdot 10^{-4}$ \\
		\midrule
		\Block{4-1}{\rotate deblur}
		& implicit & $142$ & $2.5\cdot10^3$ & $2\cdot10^3$ & $0.18$ &  - & $2\cdot 10^{-4}$   \\
		& PDPS + identity & $4.9\cdot 10^2$ & $1$ & $1$ & $0.22$ & $1\cdot 10^{-3}, 1\cdot 10^{-3}$ & $5\cdot 10^{-7}$  \\
		& PDPS + block-GS & $5.2\cdot 10^5$ & $1$ & $1$ & $0.21$ & \cref{def:inv-theta-x}, $0.1$ & $1\cdot 10^{-5}$  \\
		& trust region & $48$ & $2.5\cdot10^4$ & - & 0.18 & - & - \\
		\bottomrule
	\end{NiceTabular}
\end{table}

Our algorithm implementations are provided on Zenodo \cite{suonpera2024codes}. The specific parameter choices
for both experiments are listed in \cref{tab:setup}. To pick inner step lengths $\tau_x, \tau_y>0$  satisfying $\tau_x L/2 + \tau_x\tau_y\norm{D}^2 \le 1$ for PDPS we use the upper bound of $\norm{D}^2 \le 8$ from \cite{chambolle2004algorithm}. All the other step lengths, numbers of step for solving  inner problem and adjoint equation to a high precision  in implicit method, and parameters, such as $\omega = 1,$ are chosen by trial and error to obtain an apparently stable, but as efficient as possible, algorithm.
We do not attempt to verify their theoretical conditions. The chosen initial outer iterates $\alpha^0$ are specified in \cref{sec:mri,sec:deblur}. Initial inner and adjoint iterates for PDPS + block-GS and PDPS + identity are obtained by solving inner problem and adjoint equation to a high precision similarly than in implicit method for corresponding experiment.

The adjoint equation in the implicit method is solved with block Gauss–Seidel, see \cref{ex:splitting:block-GS}, for MRI and with the conjugate gradients squared method \cite{sonneveld1989cgs} for deblurring. For the latter, we use Matlab's \verb|cgs| implementation with tolerance $10^{-4}$ and maximum iteration count $2000.$ The dimension of the adjoint equation in our MRI experiment is large, $64911600.$ Matlab's standard linear solvers, like \verb|cgs| or \verb|bicgstab| \cite{van1992bi}, did not scale well to such a high-dimensional problem, so we used block Gauss–Seidel instead. We took 200 steps of the latter, which  still provides a less precise solution of the adjoint equation than \verb|cgs| for deblurring.

The parameters of the trust region method in the deblurring experiment are also chosen by trial and error, as follows:
\begin{itemize}[nosep]
	\item initial and maximum trust region radii $\Delta^0 = 0.01$ and $\Delta_{\max} = 0.1$, as well as rates of change $\gamma_{\text{inc}}=2$ and $\gamma_{\text{dec}}=0.1,$
	\item step validation parameters $\eta_1=0.01, \eta_2=0.5$ and $\eta'_1=0.004,$
	\item FISTA parameters for the inner problem, $\mu = 1\cdot10{-4}$, $L=1\cdot10^5$, and $\tau=1/L = 1\cdot10^{-5}.$ On each outer step, the inner FISTA is initialized using a restarting strategyfrom the previous outer step.
	It takes maximum of 25000 steps or until the tolerance $\norm{\text{inner gradient}}^2 < 10^3(\Delta^k)^2\mu^2$ is achieved.
\end{itemize}
We do not present parameters and results for the trust region method for the MRI experiment, as it scaled poorly to such a large-dimensional problem. Already the first step of the method surpassed the limit 15000 of computational resources (CPU Time), common to all algorithms.

To compare algorithm performance, we plot relative errors and values of outer objective as function of the \verb!cputime! value of Matlab on an AMD Ryzen 5 5600H CPU. We call this value “computational resources”, since it takes measures the use of several CPU cores by Matlab's internal linear algebra. This is fairer than the elapsed real time.

We need estimates $\tilde\alpha$ and $\tilde u$ of optimal $\opt\alpha$ and $\opt u=S_u(\opt\alpha)$ to compare performance of the algorithms.
For the MRI experiment these estimates are obtained by running PDPS + block-GS until computational resource (CPU time) value of 8000.
For deblur experiment er use slightly different estimates for PDPS + block-GS and PDPS + identity, which are obtained by running the corresponding algorithms for CPU time values of 4000 and 15000. The first estimates are also used to track performance of implicit method.
With these solution estimates we define the inner and outer relative errors

\[
e_{\alpha,\text{rel}} \defeq \frac{\norm{\tilde\alpha - \alpha^k}}{\norm{\tilde\alpha}}
\quad\text{and}\quad
e_{u,\text{rel}} \defeq \frac{\norm{\tilde u - u^k}_Q}{\norm{\tilde u}_Q}
\,\text{ for }\,
Q :=
\begin{pmatrix}
	\inv\tau_x\Id & - D^* \\
	- D & \inv\omega\inv\tau_y\Id
\end{pmatrix}
.
\]

\subsection{Results}
\label{sec:results}

\begin{figure}[t]
	\centering
	\begin{subfigure}[t]{0.28\textwidth}%
		\centering
		\includegraphics[width=\textwidth]{images/MRI_brain_z2_train.png}
	\end{subfigure}%
	\hspace{1cm}
	\begin{subfigure}[t]{0.28\textwidth}%
		\centering
		\includegraphics[width=\textwidth]{images/MRI_brain_b2_train.png}
	\end{subfigure}%
	\hspace{1cm}
	\begin{subfigure}[t]{0.28\textwidth}
		\centering
		\includegraphics[width=\textwidth]{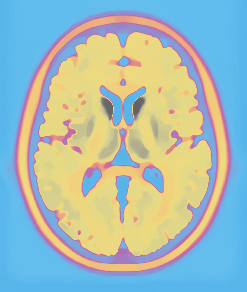}
	\end{subfigure}%
	\\
	\begin{subfigure}[t]{0.28\textwidth}%
		\centering
		\includegraphics[width=\textwidth]{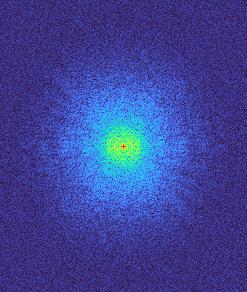}
	\end{subfigure}%
	\hspace{1cm}
	\begin{subfigure}[t]{0.28\textwidth}%
		\centering
		\includegraphics[width=\textwidth]{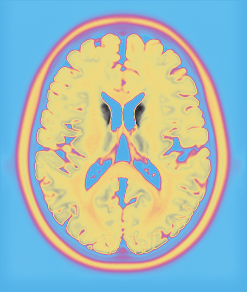}
	\end{subfigure}%
	\hspace{1cm}
	\begin{subfigure}[t]{0.28\textwidth}
		\centering
		\includegraphics[width=\textwidth]{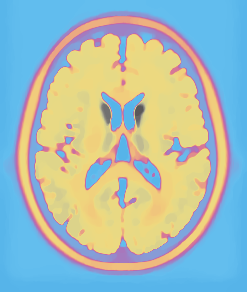}
	\end{subfigure}%
	\caption{Target images of size $247 \times 292$ (middle  		   	  column),simulated
		MRI measurements (left column) and reconstructions (right column).
		Colorbar for image domain \includegraphics[scale=0.5]{images/c_bar_MRI.png}
		values from interval $[0,1]$ and for measurements presented in logarithmic scale computed as $\log_{10}(\freevar + 0.05)$  \includegraphics[scale=0.5]{images/c_bar_FFT.png} values (of logs) from interval $[-1.3, 2.0].$ Top row presents a training set example and bottom row test example, not used in training. The relative error $\norm{\text{reconstruction} - \text{target} }_2/\norm{\text{target}}_2$ for the reconstruction of training target is $6.92\%$ and for the test target $6.86\%.$
	}
	\label{fig:MRI-2}
\end{figure}

We report performance in \cref{fig:numerical:performace} and the image data and reconstructions in \cref{fig:MRI-2,fig:deblurring}.
\Cref{fig:numerical:performace} indicates that PDPS + block-GS has significantly lower computational costs than the other methods. It has the fastest convergence of the outer objective function as well as the inner and the outer iterates in both experiments, MRI and deblurring. PDPS + identity also seems to perform better than implicit function and trust region methods based on \cref{fig:numerical:performace}, but its efficiency appears to be closer to them than to PDPS + block-GS.
This suggests that block Gauss-Seidel for adjoint steps significantly improves algorithm performance.
This improvement is likely based on fact that adjoint steps based on the block Gauss-Seidel and identity splitting have a similar computational cost, but the former allowed us to choose a 10--20 times larger outer step length parameter, see \cref{tab:setup}, while still obtaining an apparently stable algorithm.

All single-loop bilevel methods (as far as we know) require the inner objective to be twice continuously differentiable, including the methods we propose here.
This means that nonsmooth inner problems need to be smoothed.
This prevents us from (fairly) comparing our proposed method that uses PDPS steps for the inner problem, to most single-loop bilevel in the literature---including the FIFB and FEFB of \cite{suonpera2022bilevel}. Based on gradient steps for the inner problem, these would require a different type of smoothing.
Regardless, in our previous work \cite{suonpera2022bilevel}, for roughly the same deconvolution kernel identification problem that we treat here, we compared the performance of inner methods based on a single gradient descent step to the implicit method, based on the primal form of the inner problem.
The implicit method present here, based on the primal-dual form of the inner problem, performs better than the implicit method in \cite{suonpera2022bilevel}, but also, our single-loop method based on the block-GS adjoint solver, obtains greater improvements over the implicit method than the improvements obtained in \cite{suonpera2022bilevel}.

From machine learning perspective, our MRI training set is extremely small, only 4 examples.
To assess the quality of the identified sampling mask, we solved the inner inverse problem \cref{eq:numerical:mri:inner} to a high precision with PDPS (3000 steps with the step length parameters from \cref{tab:setup}, since taking more steps did not decrease the inner objective value significantly) for simulated measurements $z_i$ in a training set and in a test set of one data pair.
\Cref{fig:MRI-2} shows these reconstructions as well as the true target images.
These reconstructions lose target detail but are still of good quality (less than $7\%$ relative error) for a sampling sparsity of $28\%$.
The errors between the test set and the training set are similar, suggesting that our bilevel learning model generalises well to unseen data (from the same distribution) despite the small training set. We note that the reconstruction can be slightly improved by setting, as a post-processing step, all the non-zero learned weights to one.

As could be expected from our sparsity regularisation, for MRI, the reconstructed sampling mask in the $k$-space, shown in \cref{fig:learned-mask} (top), takes a range of values. It emphasises the lowest frequencies, and then falls of until it reaches a threshold frequency, after which the mask is zero. The threshold frequencies appear to be again emphasised to compensate for the lack of sampling at higher frequencies. The Fourier transform of the mask, which would correspond to spatial regularisation of the image, resembles a sinc function, as would be expected from the structure of the $k$-space mask, but also, interestingly, the weights of a discretised higher-order differential operator. It, therefore, seems that an optimal MRI sampling mask attempts to extract information about the differentials of the image.

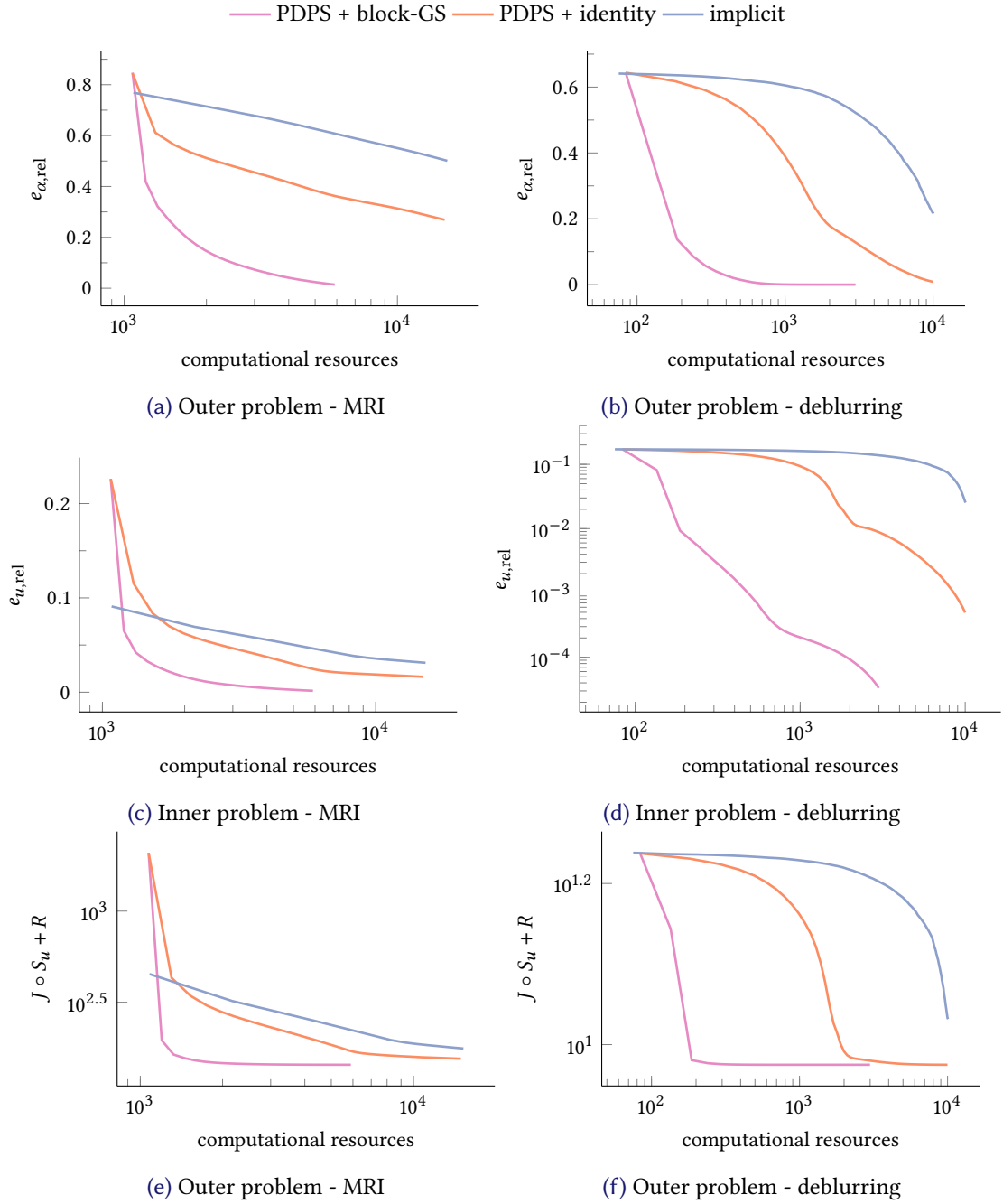
\begin{figure}[ht!]
	\centering
	\input{Tikz_image1.tex}
	\caption{Performance of the compared methods. Top row: outer iterate convergence, middle row: inner iterate convergence and bottom row: outer function value convergence.
	The “computational resources” is the spent CPU time over multiple cores, parallelisation level depending on algorithm.
	}
	\label{fig:numerical:performace}
\end{figure}

\appendix

\section{Three-point monotonicity}
\label{sec:monotonicity}

We needed the following three-point monotonicity property in \cref{lemma:outer-problem-monotonicity}.

\begin{theorem}
    \label{thm:stongly-convex-monotonicity}
    Let $X$ be Banach space and let $F: X\to\R$ be strongly convex with factor $\gamma$ (in a set $A\subset X$) as well as Lipschitz differentiable with constant $L>0$ (in $A$). Then for any $t>0,$
    \begin{equation*}
        \begin{aligned}[t]
            \iprod{DF(z)-DF(\opt x)}{x - \opt x}
            \geq
            \frac{\gamma - tL}{2}\norm{x-\opt x}^2
            +
            \frac{\gamma - tL}{2}\norm{z-\opt x}^2
            - \frac{L}{4t}\norm{x-z}^2
            \quad (\text{for all } x,z,\opt x\in A)
            .
        \end{aligned}
    \end{equation*}
\end{theorem}
\begin{proof}
	The proof is adapted from \cite{tuomov-proxtest}.
    We have
	\begin{equation}
		\label{ineq:strong-monotonicity-main}
		\begin{aligned}[t]
			\iprod{DF(z)-DF(\opt x)}{x - \opt x}
			=\,
			&
			\frac{1}{2}\iprod{DF(z)-DF(\opt x)}{x - \opt x}
			+
			\frac{1}{2}\iprod{DF(z)-DF(\opt x)}{x - \opt x}
			\\
			\geq\,
			&
			\frac{1}{2}\Bigl(
			\iprod{DF(x)-DF(\opt x)}{x - \opt x}
			+
			\iprod{DF(z)-DF(x)}{x - \opt x}
			\Bigr)
			\\
			&
			+
			\frac{1}{2}\Bigl(
			\iprod{DF(z)-DF(\opt x)}{z - \opt x}
			+
			\iprod{DF(z)-DF(\opt x)}{x - z}
			\Bigr)
			\\
			=:\,
			&
			\frac{1}{2}I_1 + \frac{1}{2}I_2.
		\end{aligned}
	\end{equation}
	The strong monotonicity and Lipschitz continuity of $DF$ with Young's inequality yield
	\begin{equation*}
		\begin{aligned}[t]
			I_1
			&
			= \iprod{DF(x)-DF(\opt x)}{x - \opt x}
			+
			\iprod{DF(z)-DF(x)}{x - \opt x}
			\\
			&
			\geq
			\gamma\norm{x-\opt x}^2
			- \norm{DF(z) - DF(x)}\norm{x - \opt x}
			\geq
			\gamma\norm{x-\opt x}^2
			- tL\norm{x-\opt x}^2 - \frac{L}{4t}\norm{x - z}^2.
		\end{aligned}
	\end{equation*}
	Inserting this and an analogous estimate for $I_2$ into \cref{ineq:strong-monotonicity-main} establishes the claim.
\end{proof}

\section{Lipschitz selections of the solution map}
\label{sec:lipschitz}

We now continue from \cref{rem:GIFB:solution-map-existence} to treat the existence and differentiability of the solution map $S_u$ of the inner problem.
The next lemma is a an implicit function variant of \cite[Lemma 22.3]{clason2020introduction}.
It relaxes the invertibility restriction on $G_u$ in the standard implicit function theorem.

\begin{lemma}
    \label{lemma:lipschitz:inverse-selection}
    On Banach spaces $U$, $\AlphaSpace$, and $Y$, suppose $G: U \times \AlphaSpace \to Y$ is continuously differentiable at $(u, \alpha)$, $G(u; \alpha)=0$, and that $G_u(u; \alpha) \in \linear(U; Y)$ has a right-inverse $\pinv{G_u(u; \alpha)} \in \linear(Y; U)$.
    Then there exists a neigborhood $V_\alpha$ of $\alpha$ and a continuously differentiable $S: V_\alpha \to U$ such that $G(S(\alt\alpha); \alt\alpha)=0$ for all $\alt\alpha \in V_\alpha$, and $S'(\alpha) = \pinv{G_u(u; \alpha)} G_\alpha(u; \alpha)$.
\end{lemma}

\begin{proof}
    Let $A \defeq G_u(u; \alpha)$ and $\pinv A \defeq \pinv{G_u(u; \alpha)}$.
    Then $P \defeq \Id - \pinv A A$ is a projection into $\ker A = \ker G_u(u; \alpha)$, in particular, $AP=0$.
    We further define
    \begin{equation*}
        \bar G: U \times \AlphaSpace \to Y \times \ker G_u(u; \alpha),\qquad \bar G(\alt u; \alt \alpha) \defeq (G(\alt u; \alt \alpha), P\alt u)
        \quad\text{for all}\ (\alt u, \alt\alpha) \in U \times \AlphaSpace,
    \end{equation*}
    as well as
    \begin{equation*}
        M: Y \times \ker A \to U,\qquad  M(\alt y, \alt q) \defeq \pinv A\alt y + \alt q,
        \quad\text{for all}\ \alt y \in Y \text{ and } \alt q \in \ker A.
    \end{equation*}
    Then for all $\dir u\in U$, we have
    $
        M\bar G_u(u; \alpha)\dir u
        =\pinv A A\dir u + P \dir u=\dir u.
    $
    Thus $M$ is a left-inverse of $\bar G_u(u; \alpha)$, and consequently $\ker \bar G_u(u; \alpha) = \{0\}$.
    Since $\bar G_u(u; \alpha)\dir u=(A\dir u, P\dir u)$ for all $\dir u\in U$, similarly, for all $(\alt y, \alt q) \in Y \times \ker A$, we have
    \[
        \bar G_u(u; \alpha)M(\alt y, \alt q)
        = (A\pinv A\tilde y + A\alt q, P\pinv A \tilde y + P \alt q)
        = (A \pinv A \tilde y, P \alt q)
         = (\tilde y, \alt q),
    \]
    which shows that $M$ is also the right-inverse of $\bar G_u(u; \alpha)$ on $Y \times \ker F'(x)$.
    Hence $\bar G_u(u; \alpha)$ is bijective.
    By the implicit function theorem, e.g., \cite[Theorem 1.25]{ioffe2017variational}, a continuously differentiable $S: V_\alpha \to U$ now exists in a neighborhood $V_\alpha$ of $\alpha$ with
    \begin{align*}
        S'(\alpha)
        &
        = \bar G_u(u; \alpha)^{-1} \bar G_\alpha(u; \alpha)
        = M\bar G_\alpha(u; \alpha)
        = M(G_\alpha(u; \alpha), 0)
        \\
        &
        = \pinv{A} G_\alpha(u; \alpha)
        = \pinv{G_u(u; \alpha)} G_\alpha(u; \alpha).
        \qedhere
    \end{align*}
\end{proof}

We now obtain the following; compare \cite[Theorem 5J.8]{dontchev2014implicit}.

\begin{corollary}
    \label{cor:lipschitz:lipschitz-inverse-selection}
    Suppose that the assumptions of \cref{lemma:lipschitz:inverse-selection} hold at some $u_\alpha$ with $G(u_\alpha; \alpha)=0$ for all $\alpha$ in a compact $\breve\AlphaSpace \subset \AlphaSpace$.
    If, moreover
    \begin{enumerate}[label=(\roman*),nosep]
        \item\label{item:lipschitz:lipschitz-inverse-selection:convex}
        The set $\{u \mid G(u; \alpha) = 0\}$ is convex for all $\alpha \in \breve\AlphaSpace$, and
        \item\label{item:lipschitz:lipschitz-inverse-selection:lip-bound}
        $M \defeq \sup_{\alpha \in \breve\AlphaSpace} \norm{\pinv{G_u(u_\alpha; \alpha)} G_\alpha(u_\alpha; \alpha)} < \infty$,
    \end{enumerate}
    then for every $\epsilon>0$, there exists a continuously differentiable and Lipschitz $S: \breve\AlphaSpace \to U$ with $G(S(\alpha); \alpha)=0$ for all $\alpha \in \breve\AlphaSpace$.
    Moreover, the Lipschitz factor of $S$ is at most $M + \epsilon$.
\end{corollary}

\begin{proof}
    We apply \cref{lemma:lipschitz:inverse-selection} to all $\alpha \in \breve\AlphaSpace$ to obtain corresponding neighbourhoods $V_\alpha \subset \breve\AlphaSpace$ and continuously differentiable mappings $S^\alpha: V_\alpha \to U$ satisfying $G(S^\alpha(\alt \alpha); \alt\alpha)=0$ for all $\alt\alpha \in V_\alpha$. Due to continuous differentiability and the expression $(S
    ^\alpha)'(\alpha) = \pinv{G_u(u; \alpha)} G_\alpha(u; \alpha)$, shrinking $V_\alpha$ if necessary, we may assume that $S^\alpha$ is Lipschitz in $V_\alpha$ with Lipschitz factor at most $\norm{\pinv{G_u(u; \alpha)} G_\alpha(u; \alpha)} + \epsilon$.

    If $\alt\alpha \in V_{\alpha_1} \isect V_{\alpha_2}$ for $\alpha_1 \ne \alpha_2$, it may be that $S^{\alpha_1}(\alt\alpha) \ne S^{\alpha_2}(\alt\alpha)$.
    However, the convexity assumption \cref{item:lipschitz:lipschitz-inverse-selection:convex} guarantees for any $\lambda \in [0, 1]$ that $u^\lambda \defeq \lambda S^{\alpha_1}(\alt\alpha) + (1-\lambda) S^{\alpha_2}(\alt\alpha)$ satisfies $G(u^\lambda; \alt\alpha)=0$.
    Therefore, we may use a compact covering argument and a partition of unity (which exists by the compactness of $\breve\AlphaSpace$) to glue together the various $S^\alpha$ to obtain a single inverse selection $S: \breve\AlphaSpace \to U$, satisfying $G(S(\alpha); \alpha)=0$ for all $\alt\alpha$.
    Finally, \cref{item:lipschitz:lipschitz-inverse-selection:lip-bound} establishes the claimed bound on the Lipschitz factor.
\end{proof}

If $G(u; \alpha) = \grad_u F(u; \alpha)$ for a convex function $F$, then it is clear that the convexity assumption in \cref{cor:lipschitz:lipschitz-inverse-selection}\,\cref{item:lipschitz:lipschitz-inverse-selection:convex} holds.
However, it also holds for $G(u; \alpha) = A(u; \alpha) + B(u; \alpha)$ as given by \cref{ex:pdps} for the PDPS.
As a variation of \cite[Theorem 2.1]{he2012convergencerate}, this can be seen from the characterisation of solutions $(\bar x, \bar y)$ to $0 = G(\bar x, \bar y; \alpha)$ as the saddle points of $\mathcal{L}(x, y) \defeq f(x; \alpha) + \iprod{Kx}{y} - g^*(y; \alpha)$.
The latter as defined as those $(\bar x, \bar y)$ satisfying
\[
    \mathcal{L}(\bar x, y) \le \mathcal{L}(\bar x, \bar y) \le \mathcal{L}(x, \bar y)
    \quad\text{for all}\quad (x, y) \in X \times Y.
\]

\begin{lemma}
    The saddle points of a convex-concave $\mathcal{L}: X \times Y \to [-\infty,\infty]$ form a convex set.
\end{lemma}

\begin{proof}
    For any two solutions pairs $(x_0, y_0)$ and $(x_1, y_1)$, defining $x_\lambda \defeq \lambda x_0 + (1-\lambda) x_1$ and $y_\lambda \defeq \lambda y_0 + (1-\lambda) y_1$, by convexity
    $
        \mathcal{L}(x_\lambda, y_\lambda)
        \le
        \lambda \mathcal{L}(x_0, y_\lambda) + (1-\lambda) \mathcal{L}(x_1, y_\lambda)
        \le
        \mathcal{L}(x, y_\lambda)
    $
    for any $x$. The lower bound is proved analogously.
\end{proof}

\section{Lemmas for the tracking property of block Gauss--Seidel}
\label{sec:block-gs}

We prove here an estimate that was needed to derive the tracking inequality for block Gauss–Seidel splitting in \cref{ex:splitting:block-GS}. We start with two technical lemmas.

\begin{lemma}
	\label{lemma:block-GS-helper}
	Let $b > 0$ and $c \le 0.$ Then
    $
		-b + \sqrt{b^2 - 4 bc} \ge  -2bc/(-c + b).
    $
\end{lemma}

\begin{proof}
	We estimate
	\begin{align*}
			-b + \sqrt{b^2 - 4 bc}
			&
			=
			\frac{b^2 - 4 bc - b^2}{b + \sqrt{b^2 - 4 bc}}
			=
			\frac{-4bc}{b + \sqrt{b^2 - 4 bc}}
			\ge
			\frac{-4bc}{b + \sqrt{b^2 - 4 bc + (2c)^2}}
			=
			\frac{-2bc}{-c + b}. \qedhere
	\end{align*}
\end{proof}

\begin{lemma}
	\label{lemma:block-GS-eta}
	Let $a,b,c,d > 0$ and define for $\eta_1, \eta_2>0$ the functions
	\[
		f_1(\eta_1, \eta_2) = (1 + \inv\eta_1 +(1 + \inv\eta_2)a^2)b^2
        \quad\text{and}\quad
		f_2(\eta_1, \eta_2) = (1 + \eta_1)c^2 +(1 + \eta_2)d^2.
	\]
	Moreover, assume for some $\zeta\in[0,1)$ that
	\begin{equation}
		\label{ineq:adjoint-zeta-general}
		d^2 + abd + b^2(1+c)(1+a^2) + c^2 + bc
		\le \zeta^2
	\end{equation}
	Then there exists $\eta_1^*>0$ such that
	\begin{equation*}
		f_1(\eta_1^*, ac\inv d\eta_1^*)
		=
		f_2(\eta_1^*, ac\inv d\eta_1^*)
		\le
		\zeta^2.
	\end{equation*}
\end{lemma}
\begin{proof}
	The equation $f_1(\eta_1^*, ac\inv d\eta_1^*) = f_2(\eta_1^*, ac\inv d\eta_1^*)$ is equivalent to the quadratic
	\begin{equation*}
		g(c^2 + acd)(\eta_1^*)^2 + (c^2+d^2 - b^2(1 + a^2))\eta_1^* - b^2/c^2(c^2 + acd) = 0.
	\end{equation*}
	The left-hand-side is negative at $\eta_1^*=0$, and the equation has a root $\eta_1^*>0$.
    Instead of solving for this root, we find the root greater than $\eta_1^*$ of
    \[
        q(\eta) = (c^2 + acd)\eta^2 - b^2(1 + a^2)\eta - b^2/c^2(c^2 + acd).
    \]
    This is
	\begin{equation*}
        \begin{aligned}[t]
		\eta
        &
        = \frac{
			b^2(1 + a^2) + \sqrt{(b^2(1 + a^2))^2 + 4b^2/c^2(c^2 + acd)^2}
		}{
			2(c^2 + acd)
		}
        \\
        &
		\le
		\frac{
			b^2(1 + a^2) + b/c(c^2 + acd)
		}{
			c^2 + acd
		}
		=
		\frac{b}{c}r
        \quad\text{for}\quad
        r \defeq
		\left(
			1 + \frac{b(1+a^2)}{c + ad}
		\right).
        \end{aligned}
	\end{equation*}
	Because $f_2(\eta_1, (ac/d)\eta_1)$ is an increasing function of $\eta_1,$ an application of \cref{ineq:adjoint-zeta-general} gives
	\begin{align*}
		f_2(\eta_1^*, (ac/d)\eta_1^*)
		&
		\le f_2(\eta, (ac/d)\eta)
		=
		\left(
			1 +
			\frac{b}{c}
            r
		\right)c^2
		+
		\left(
			1 +
			\frac{ab}{d}
            r
		\right)d^2
		\\
		&
		\le
		d^2 + abd + b^2(1+c)(1+a^2) + c^2 + bc
		\le \zeta^2.
		\qedhere
	\end{align*}
\end{proof}

\begin{theorem}
	Assume $A_{11}, A_{12}, A_{21}$ and $A_{22}$ are linear operator between Hilbert spaces $U$ and $\AlphaSpace$.
    With $A_{11}=N_{11} + M_{11}$ and $A_{22}=N_{22} + M_{22}$, let
	\[
		A
		=
		\begin{pmatrix}
			A_{11} & A_{12} \\
			A_{21} & A_{22}
		\end{pmatrix},
		\quad
		N
		=
		\begin{pmatrix}
			N_{11} & 0 \\
			A_{21} & N_{22}
		\end{pmatrix},
		\quad
		\text{ and }
		\quad
		M
		=
		\begin{pmatrix}
			M_{11} & A_{12} \\
			0 & M_{22}
		\end{pmatrix}.
	\]
	Moreover, suppose $N_{11}$ and $N_{22}$ are invertible and bounded, and
	\begin{multline}
        \label{eq:block-gs-assumption}
		\norm{\inv N_{22}(M_{22}-A_{21}\inv N_{11}A_{12})}^2
		+
		\norm{\inv N_{22}(M_{22}-A_{21}\inv N_{11}A_{12})}\norm{\inv N_{22} A_{21}}\norm{\inv N_{11} M_{11}}
		\\
		+
		\norm{\inv N_{11} M_{11}}^2(1+\norm{\inv N_{11} A_{12}} )(1 + \norm{\inv N_{22} A_{21}}^2)
		+
		\norm{\inv N_{11} A_{12}}^2
		+
		\norm{\inv N_{11} A_{12}}
		\norm{\inv N_{11} M_{11}}
		\le
		\zeta^2
	\end{multline}
	for some $\zeta \in [0, 1).$
	Then \eqref{eq:convergence:splitting:split-condition} holds with the aforementioned $\zeta$ and any $\gamma_N$ satisfying
	\[
	0
	\leq
	\gamma_N
	\leq
	\frac{\norm{N_{11}}\norm{N_{22}}}{2\norm{N_{11}} + \norm{N_{22}}(1 + \norm{\inv N_{22}A_{21}}^2)}.
	\]
\end{theorem}
\begin{proof}
    We have
    $$
    \inv N = \begin{pmatrix}
        \inv N_{11} & 0 \\
        - \inv N_{22} A_{21} \inv N_{11} & \inv N_{22}
    \end{pmatrix}.
    $$
    Thus Young's inequality, for any $\beta > 0$, establishes
    \begin{equation}
    	\label{ineq:op-norm-Yuong}
    \begin{aligned}[t]
        \norm{\inv N}^2
        &
        = \sup_{x,y} \frac{\norm{\inv N_{11}x}^2 + \norm{\inv N_{22} y - \inv N_{22} A_{21} \inv N_{11}x}^2}{\norm{x}^2 + \norm{y}^2}
        = \sup_{x,y} \frac{\norm{x}^2 + \norm{y - \inv N_{22} A_{21} x}^2}{\norm{N_{11}x}^2 + \norm{N_{22}y}^2}
        \\
        &
        \le \sup_{x,y} \frac{\norm{x}^2 + (1+\inv\beta)\norm{y}^2 + (1+\beta)\norm{\inv N_{22} A_{21} x}^2}{\norm{N_{11}x}^2 + \norm{N_{22}y}^2}.
    \end{aligned}
    \end{equation}
    Consequently $\rho \norm{\inv N}^2 \le 1$ (here $\rho=\gamma_N^2$) if and only if, for all $x$ and $y$,
    $$
        \norm{x}^2 + (1+\beta)\norm{\inv N_{22} A_{21} x}^2
        \le \inv\rho\norm{N_{11}x}^2
        \quad\text{and}\quad
        (1+\inv\beta)\norm{y}^2 \le \inv\rho\norm{N_{22}y}^2.
    $$
    This is to say
    \begin{equation}
        \label{ineq:GS-block-rho}
        \rho + \rho(1+\beta)\norm{\inv N_{22} A_{21}}^2
        \le \norm{N_{11}}^2
        \quad\text{and}\quad
        \rho(1+\inv\beta) \le \norm{N_{22}}^2.
    \end{equation}
    We can solve from the latter
    $$
        \inv\beta = \inv\rho \norm{N_{22}}^2 - 1
        \iff
        \beta = \frac{\rho}{\norm{N_{22}}^2 - \rho}
        \iff
        \rho = \frac{\norm{N_{22}}^2}{1+\inv\beta}
        = \frac{\beta\norm{N_{22}}^2}{\beta+1}
    $$
    provided $\norm{N_{22}}^2 > \rho.$
    Then the first inequality in \cref{ineq:GS-block-rho} reads
    $$
    \frac{\beta\norm{N_{22}}^2}{\beta+1} + \frac{\beta\norm{N_{22}}^2}{\beta+1}(1+\beta)\norm{\inv N_{22} A_{21}}^2
    \le \norm{N_{11}}^2.
    $$
    In other words
    \begin{equation*}
        \beta\norm{N_{22}}^2 + (1+\beta)\beta\norm{N_{22}}^2\norm{\inv N_{22} A_{21}}^2
        \leq
        (1+\beta)\norm{N_{11}}^2.
    \end{equation*}
    Note that this holds with $\beta=0$, and rearranges as
    \begin{equation}
        \label{eq:proof-block-beta-2nd-order}
        \beta^2\norm{N_{22}}^2\norm{\inv N_{22} A_{21}}^2
        +
        \beta\left(
            \norm{N_{22}}^2(1 + \norm{\inv N_{22} A_{21}}^2) - \norm{N_{11}}^2
        \right)
        -
        \norm{N_{11}}^2
        \le 0.
    \end{equation}
    We want maximal $\rho$, therefore minimal $\beta>0$ that satisfies \eqref{eq:proof-block-beta-2nd-order}.
    Since this is condition is quadratic in $\beta$, and is satisfied at $\beta=0,$, the corresponding function has a positive root $\beta'$.
    Any positive number smaller than $\beta',$ such as the positive root of
    \[
        \beta^2\norm{N_{22}}^2(1 + \norm{\inv N_{22} A_{21}}^2)
        +
        \beta\norm{N_{22}}^2(1 + \norm{\inv N_{22} A_{21}}^2)
        -
        \norm{N_{11}}^2,
    \]
    also satisfies \cref{eq:proof-block-beta-2nd-order}.
    We solve the corresponding quadratic equation in $\beta$ and use \cref{lemma:block-GS-helper} to find $\beta^*<\beta'$
    \begin{multline*}
        \frac{-
            \norm{N_{22}}^2(1 + \norm{\inv N_{22} A_{21}}^2)
             + \sqrt{
                \norm{N_{22}}^4(1 + \norm{\inv N_{22} A_{21}}^2)^2
                -4\norm{N_{22}}^2(1 + \norm{\inv N_{22} A_{21}}^2)\norm{N_{11}}^2 }}
                {2\norm{N_{22}}^2(1 + \norm{\inv N_{22} A_{21}}^2)}
        \\
        \ge
        \frac{
            \norm{N_{11}}^2
        }{
            \norm{N_{11}}^2
            +
            \norm{N_{22}}^2(1 + \norm{\inv N_{22} A_{21}}^2)
        } =: \beta^* .
    \end{multline*}
    Therefore $\rho\norm{\inv N_v}^2 \le 1$ for $\rho = \beta^*\norm{N_{22}}^2/(1 + \beta^*)$ and we get
    $$
        \rho
        =
        \frac{\beta^*\norm{N_{22}}^2}{\beta^*+1}
        =
        \frac{
            \norm{N_{11}}^2  \norm{N_{22}}^2
        }{
            2\norm{N_{11}}^2
            +
            \norm{N_{22}}^2(1 + \norm{\inv N_{22} A_{21}}^2)
        }
        \ge
        \frac{
            \norm{N_{11}}^2  \norm{N_{22}}^2
        }{
            \left(
                2\norm{N_{11}}
                +
                \norm{N_{22}}(1 + \norm{\inv N_{22} A_{21}}^2)
            \right)^2
        }
    $$
    Taking the square root, we see that $\gamma_N \norm{\inv N} \le 1$ for any
    \[
        \gamma_N \le \norm{N_{11}}\norm{N_{22}}/(2\norm{N_{11}} + \norm{N_{22}}(1 + \norm{\inv N_{22}A_{21}}^2)).
    \]

    Likewise we have
    $$
        \inv N_v M_v
        =
        \begin{pmatrix}
            \inv N_{11} M_{11} & \inv N_{11} A_{12} \\
            -\inv N_{22}A_{21}\inv N_{11} M_{11} & -\inv N_{22} A_{21} \inv N_{11} A_{12} + \inv N_{22} M_{22}
        \end{pmatrix}.
    $$
    Similarly as in \cref{ineq:op-norm-Yuong}, using Young's inequality we obtain
    $
        \norm{\inv N_v M_v}^2
        \le \zeta^2
    $
    if and only if, for some $\eta_1, \eta_2 > 0$ and $\zeta \in [0, 1)$, we have
    \begin{equation*}
        \begin{aligned}
                \left(
                    1 + \inv\eta_1
                    +
                    (1 + \inv\eta_2)\norm{\inv N_{22} A_{21}}^2
                \right)
                \norm{\inv N_{11} M_{11}}^2
                &
                \le
                \zeta^2
                \quad\text{and}
                \\
                (1+\eta_1)\norm{\inv N_{11} A_{12}}^2
                +
                (1+\eta_2)\norm{\inv N_{22}(M_{22}-A_{21}\inv N_{11}A_{12})}^2
                &
                \le
                \zeta^2.
            \end{aligned}
    \end{equation*}
    By \cref{lemma:block-GS-eta} these inequalities hold when the assumed \eqref{eq:block-gs-assumption} does.
\end{proof}

\section{Proximal operators}
\label{sec:prox}

\subsection{Outer proximal operators}
\label{sec:outer-prox-operators}

\begin{lemma}
    \label{lemma:prox:simple}
    Let $R(\alpha)=\beta(\sum_{i=2}^4\alpha_i-1)^2+\delta_{[0, \infty)}(\alpha_1)$
    for $\alpha=(\alpha_1,\ldots,\alpha_4) \in \R^4$.
    Then
    \[
        \prox_{ \sigma R }(\tilde{a})=( \max\{0,\tilde\alpha_1\}, \alpha_2,\alpha_3,\alpha_4)
    \]
    where $\alpha_2,\alpha_3$, and $\alpha_4$ are solved from the system of linear equations
    \[
        \alpha_j + 2\sigma\beta \sum_{i=2}^4 \alpha_j = \tilde{a}_j + 2\sigma\beta.
    \]
\end{lemma}
\begin{proof}
	The claim follows directly from the definition of proximal operator and the corresponding first order optimality condition
	\begin{equation*}
		\partial_{\alpha_j} \left[ \sigma\beta(\sum_{i=2}^4\alpha_i-1)^2  + \frac{1}{2}\sum_{i=2}^4(\alpha_i-\tilde\alpha_i)^2 \right] = 0
        \quad\text{for } j = 2, 3, 4.
        \qedhere
	\end{equation*}
\end{proof}

The next result on our sparsity regulariser for MRI sampling pattern weights has a lot of resemblance to simplex projection methods, see, for example, \cite[Chapter 5]{angerhausen2022stochastic}.

\begin{lemma}
    \label{lemma:prox:isimplexl1}
    For some weight vector $w \in (0, \infty)^n$ and $\beta, M>0$, let
    \[
        R(\alpha) \defeq \beta\phi(w^\top\alpha) + \delta_{[0, \infty)^n}(\alpha), \quad \phi(t) = t + \delta_{(-\infty, M]} (t).
    \]
    Then, for all $i=1,\ldots,n$ and $\tilde\alpha \in \R^n$,
    \[
        [\prox_{\tau R}(\tilde\alpha)]_i = \max(0, {\tilde\alpha}_i - w_i \tilde\lambda)
        \quad\text{where}\quad
        \tilde\lambda = \max\left\{\frac{w_{\mathcal{I}}^\top {\tilde\alpha}_{\mathcal{I}} - M}{\norm{w_{\mathcal{I}}}^2}, \tau\beta \right\},
    \]
    where $\tau > 0, {\tilde\alpha}_{\mathcal{I}}=(\tilde\alpha_i)_{i \in \mathcal{I}}$ and $w_{\mathcal{I}}=(w_i)_{i \in \mathcal{I}}$ for  $\mathcal{I} \subset \{1,\ldots,n\}$ chosen such that
    \begin{equation}
        \label{eq:prox:isimplexl1:tcond}
        \tilde\alpha_i/w_i > \tilde\lambda
        \implies i \in \mathcal{I}
        \quad\text{and}\quad
        \tilde\alpha_i/w_i < \tilde\lambda
        \implies
        i \not\in \mathcal{I}.
    \end{equation}
    Moreover, such a choice exists.
\end{lemma}

Thus, due to \eqref{eq:prox:isimplexl1:tcond}, to calculate $\alpha$, we need to sort the vector $(\alpha_i/w_i)_{i=1}^n$, and find a dividing index $j$ such that $\mathcal{I}=\{j+1,\ldots,n\}$ and $\tilde\lambda$ computed based on it satisfies \eqref{eq:prox:isimplexl1:tcond}. If $\alpha_i/w_i=\alpha_j/w_j$ and $w_i \ne w_j$ for $i \ne j$, then all sorting combinations may be needed to be tried for such indices.
If $w^\top\alpha < M$, we can always take $\mathcal{I}=\{1,\ldots,n\}$.

\begin{proof}
    The necessary and sufficient first-order optimality condition for $\alpha=\prox_{\tau R}(\tilde\alpha)$ is
    \[
        0 \in \alpha - {\tilde\alpha} + w N_{\le M}(w^\top \alpha) + \tau\beta w + N_{\ge 0}(\alpha),
    \]
    in other words
    \begin{gather}
        \label{eq:prox:isimplexl1:0}
        0 = \alpha - {\tilde\alpha} + w(\lambda + \tau\beta)
        + \mu
        \quad\text{where}
        \\
        \label{eq:prox:isimplexl1:0lambda}
        0 \le \lambda, \quad w^\top\alpha \le M, \quad (M-w^\top\alpha)\lambda = 0,
        \quad\text{and}
        \\
        \label{eq:prox:isimplexl1:0mu}
        \mu_i \le 0, \quad \alpha_i \ge 0, \quad \alpha_i\mu_i = 0
        \quad\text{for all}\quad i=1,\ldots,n.
    \end{gather}

    We choose
    \[
        \lambda = \tilde\lambda - \tau\beta
        = \max\left\{\frac{w_{\mathcal{I}}^\top {\tilde\alpha}_{\mathcal{I}} - M}{\norm{w_{\mathcal{I}}}^2} - \tau\beta, 0\right\}.
    \]
    For any $i \in \{1,\ldots,n\}$,
    if $\tilde\alpha_i \ge w_i\tilde\lambda$, then $\alpha_i=\tilde\alpha_i - w_i\tilde\lambda \ge 0$, so \eqref{eq:prox:isimplexl1:0mu} is satisfied by taking $\mu_i=0$, and  \eqref{eq:prox:isimplexl1:0} when $\alpha_i - {\tilde\alpha}_i + w_i \lambda + \tau\beta w_i = 0$, which indeed holds.
    Likewise, if $\tilde\alpha_i < w_i\tilde\lambda$, we have $\alpha_i=0$, so \eqref{eq:prox:isimplexl1:0mu} is satisfied for any $\mu_i \le 0$, and, therefore, \eqref{eq:prox:isimplexl1:0} when $\alpha_i - {\tilde\alpha}_i + w_i \lambda + \tau\beta w_i \ge 0$, which again holds.

    It, therefore, remains to verify \eqref{eq:prox:isimplexl1:0lambda}.
    To do so, we start by proving that
    \begin{equation}
        \label{eq:prox:isimplexl1:2}
        w^\top \alpha = w_{\mathcal{I}}^\top\tilde\alpha_{\mathcal{I}} - \norm{w_{\mathcal{I}}}^2\tilde\lambda,
    \end{equation}
    Indeed, by the construction $\alpha_i=\max(0, {\tilde\alpha}_i - w_i \tilde\lambda)$, we have
    \[
        w^\top\alpha = \sum_{i:{\tilde\alpha}_i \ge w_i \tilde\lambda} w_i\tilde\alpha_i - w_i^2\tilde\lambda
        = \sum_{i:{\tilde\alpha}_i > w_i \tilde\lambda} w_i\tilde\alpha_i - w_i^2\tilde\lambda.
    \]
    Therefore \eqref{eq:prox:isimplexl1:2} holds by \eqref{eq:prox:isimplexl1:tcond}.
    Now, to prove \eqref{eq:prox:isimplexl1:0lambda}, observe that $0 \le \lambda$ by construction.
    Using \eqref{eq:prox:isimplexl1:2}, and the definition of $\tilde\lambda$, we obtain
    $
        w^\top\alpha = \min\bigl(M, w_{\mathcal{I}}^\top\tilde\alpha_{\mathcal{I}}-\norm{w_{\mathcal{I}}}^2\tau\beta \bigr).
    $
    In particular $w^\top \alpha \le M$.
    Moreover, if $\lambda > 0$,
    i.e., $w_{\mathcal{I}}^\top\tilde\alpha_{\mathcal{I}} - \norm{w_{\mathcal{I}}}^2\tau\beta > M$, we obtain
    $w^\top\alpha=M$, proving \eqref{eq:prox:isimplexl1:0lambda}.

    To see that \eqref{eq:prox:isimplexl1:tcond} can be satisfied, let $\mathcal{J} \defeq \{i \mid \mu_i = 0\}$.
    Observe from \eqref{eq:prox:isimplexl1:0mu} that $\mu_i=0$ whenever $\alpha_i>0$.
    Since $w_{\mathcal{J}}^\top \alpha_{\mathcal{J}} = w^\top \alpha \le M$ and $w_{\mathcal{J}}^\top\mu_{\mathcal{J}}$, applying $w_{\mathcal{J}}^\top$ to the components of \eqref{eq:prox:isimplexl1:0} corresponding to $\mathcal{J}$, we obtain
    \begin{equation}
        \label{eq:prox:isimplexl1:3}
        0 \le M - w_{\mathcal{J}}^\top \tilde\alpha_{\mathcal{J}} + \norm{w_{\mathcal{J}}}^2 (\lambda + \tau\beta).
    \end{equation}
    Due to this and the first part of \cref{eq:prox:isimplexl1:0lambda},
    \[
        \lambda \ge \max\left\{0, \frac{w_{\mathcal{J}}^\top \tilde\alpha_{\mathcal{J}} - M}{ \norm{w_{\mathcal{J}}}^2 } - \tau\beta\right\}.
    \]
    If $w^\top \alpha < M$, we need $\lambda=0$, so this must hold as an equality.
    If $w^\top\alpha=M$, the inequality in \eqref{eq:prox:isimplexl1:3} would be an equality, so we again reach the same conclusion due to the non-negativity requirement.
    It follows from \cref{eq:prox:isimplexl1:0,eq:prox:isimplexl1:0mu} that
    \[
        \alpha_i = {\tilde\alpha}_i - w_i(\lambda + \tau\beta) - \mu_i
        = {\tilde\alpha}_i - w_i \tilde \lambda - \mu_i
        = \max\{0, {\tilde\alpha}_i - w_i \tilde \lambda_{\mathcal{J}}\},
    \]
    where
    \[
        \tilde\lambda_{\mathcal{J}} = \lambda + \tau\beta = \max\left\{\frac{w_{\mathcal{J}}^\top \tilde\alpha_{\mathcal{J}} - M}{ \norm{w_{\mathcal{J}}}^2 }, \tau\beta\right\}.
    \]
    Finally, if $\tilde\alpha_i > w_i\tilde\lambda_{\mathcal{J}}$ then $i \in \mathcal{J}$ by definition. Since also $\tilde\alpha_i<w_i\tilde\lambda_{\mathcal{J}}$ implies $\mu_i<0$, hence $i \not \in \mathcal{J}$, it follows that we can take $\mathcal{I}=\mathcal{J}$ and $\tilde\lambda=\tilde\lambda_{\mathcal{J}}$.
\end{proof}

\subsection{Differentials and the proximal operator of the inner regularizer}
\label{sec:inner-prox-operator}

We present here expansions of derivations and the proximal operator of the inner regularizer $g_{\varepsilon, \delta}^*$ defined in \cref{def:inner-reg}.
These are needed for the numerical realisation of our method for the experiments of \cref{sec:numerical}.
Readily, we calculate the partial derivatives of $g_{\varepsilon, \delta}^*$ as (note that each $y_i\in\R^2$)
\begin{equation}
	\label{eq:partial-inner-reg}
	\begin{aligned}[t]
        \partial_{y_j}g_{\varepsilon, \delta}^*(y)
		&
		=
		\partial_{y_j}\left[ \max\{0,\frac{1}{3\varepsilon}(\norm{y_j} - \alpha_0)^3\} + \frac{\delta}{2}\norm{y_j}^2\right]
        \\
		&
		=
		\delta y_j +
		\begin{cases}
			0, & \text{if $\norm{y_j} < \alpha_0$}\\
			\frac{1}{\varepsilon}\frac{y_j}{\norm{y_j}}(\norm{y_j} - \alpha_0)^2, & \text{otherwise}.
		\end{cases}
	\end{aligned}
\end{equation}
Likewise, the Hessian $\grad_y^2 g_{\varepsilon, \delta}^*(y)$ is given by the entries $\partial_{y_j}\partial_{y_i} g_{\varepsilon, \delta}^*(y) = 0$ for $i \neq j$ and
\begin{align*}
	\partial^2_{y_j} g_{\varepsilon, \delta}^*(y)
	&
	=
	\delta \Id_2 +
	\begin{cases}
		0, & \text{if $\norm{y_j} < \alpha_0$} \\
		\partial_{y_j}\left[\frac{1}{\varepsilon}\frac{y_j}{\norm{y_j}}(\norm{y_j} - \alpha_0)^2\right], & \text{otherwise}
	\end{cases} \\
	&
	=
	\delta \Id_2 +
	\begin{cases}
		0, & \text{if $\norm{y_j} < \alpha_0$} \\
		\frac{1}{\varepsilon}\frac{\Id_2 - y_jy_j^T\norm{y_j}^{-2}}{\norm{y_j}}
		(\norm{y_j} - \alpha_0)^2
		+ \frac{1}{\varepsilon}\frac{2y_jy_j^T}{\norm{y_j}^2}(\norm{y_j} - \alpha_0)
		, & \text{otherwise.}
	\end{cases}
\end{align*}

\begin{lemma}
	Let  $g_{\varepsilon, \delta}^*$ be defined by \cref{def:inner-reg} for some $\varepsilon, \delta > 0$.
    Moreover, let $v\in\R^{2n_1n_2}$.
    Then for any $\tau > 0$ and $j \in \{1,\ldots,n_1n_2\}$, we have
	\begin{equation*}
		[\prox_{ \tau g_{\varepsilon, \delta}^* }(v)]_j
		=
		\begin{cases}
			\frac{1}{1 + \tau \delta}v_j, & \text{if $\norm{v_j} < (1+\tau\delta)\alpha_0$} \\
			\frac{2\alpha_0 - \varepsilon(\inv\tau + \delta) + \sqrt{\varepsilon^2(\inv\tau + \delta)^2 + 4\varepsilon(\inv\tau\norm{v_j} - (\inv\tau + \delta) \alpha_0)}}{2\norm{v_j}}v_j
			, & \text{otherwise.}
		\end{cases}
	\end{equation*}
\end{lemma}
\begin{proof}
    Let $y \defeq \prox_{ \tau g_{\varepsilon, \delta}^* }(v)$.
    Let $j \in \{1,\ldots, n_1n_2\}$.
    Suppose first that $\norm{y_j} < \alpha_0$
	Using \cref{eq:partial-inner-reg} and the the definition of proximal operator, give the characterisation
	$y_j - v_j + \tau \delta y_j = 0$, i.e.,
    $y_j = \frac{1}{1 + \tau \delta}v_j$.
	Thus
    $\norm{v_j} < (1+\tau\delta)\alpha_0.$

    Suppose then that $\norm{y_j} \ge \alpha_0$.
    Again the partial derivative formula \cref{eq:partial-inner-reg} and the definition of the proximal operator give the characterisation
	\[
		y_j - v_j + \tau\left(\delta y_j + \frac{1}{\varepsilon}\frac{y_j}{\norm{y_j}}(\norm{y_j} - \alpha_0)^2 \right) = 0.
    \]
	This is to say $y_j = t v_j$ for some $t>0$ satisfying
	\begin{equation*}
		t\left(
		1 + \tau\left(\delta + \frac{1}{\varepsilon}\frac{1}{t\norm{v_j}}(t\norm{v_j} - \alpha_0)^2 \right)
		\right)
		= 1,
	\end{equation*}
    which translates into the second order equation in $t$.
	\begin{equation*}
		\norm{v_j}t^2
		+
		(\varepsilon(\inv\tau + \delta) - 2\alpha_0)t
		+
		\alpha_0^2\inv{\norm{v_j}} - \varepsilon\inv\tau
		= 0.
	\end{equation*}
	This is solved by
	\begin{equation*}
		t = \frac{2\alpha_0 - \varepsilon(\inv\tau + \delta) \pm \sqrt{\varepsilon^2(\inv\tau + \delta)^2 + 4\varepsilon(\inv\tau\norm{v_j} - (\inv\tau + \delta) \alpha_0)}}{2\norm{v_j}}.
	\end{equation*}
    Only the positive choice of sign satisfies the ansatz $\norm{y_j} \ge \alpha_0$, i.e., $t\norm{v_j} \ge \alpha_0$.
\end{proof}

\input{bilevel_general2.bbl}
\bibliographystyle{jnsao}

\end{document}

%% file: symbols.tex
\newcommand{\term}{\emph}

\newcommand{\field}[1]{\mathbb{#1}}
\newcommand{\N}{\mathbb{N}}

\newcommand{\R}{\field{R}}
\newcommand{\extR}{\overline \R}

\newcommand{\norm}[1]{\|#1\|}

\newcommand{\inv}[1]{#1^{-1}}
\newcommand{\grad}{\nabla}
\newcommand{\freevar}{\,\boldsymbol\cdot\,}
\newcommand{\Union}\bigcup
\newcommand{\Isect}\bigcap
\newcommand{\union}\cup
\newcommand{\isect}\cap
\newcommand{\bigunion}\bigcup
\newcommand{\bigisect}\bigcap

\newcommand{\defeq}{:=}

\newcommand{\subdiff}{\partial}

\newcommand{\pinv}[1]{#1^\dag}

\DeclareMathOperator*{\argmin}{arg\,min}

\DeclareMathOperator{\Dom}{dom}

\DeclareMathOperator{\prox}{prox}

\newcommand{\iprod}[2]{\langle #1,#2\rangle}

\def \weaktostarSym{\setbox0=\hbox{$\rightharpoonup$}\rlap{\hbox 
        to\wd0{\hss\raise1ex\hbox{$\scriptscriptstyle{*\,}$}\hss}}\box0}
    
\def\linear{\mathbb{L}}

\def\extR{\overline \R}

\let\phi=\varphi
\let\epsilon=\varepsilon
\def\alt#1{\tilde #1}

\DeclareMathOperator{\Id}{Id}

\def\nexxt#1{#1^{k+1}}
\def\this#1{#1^k}
\def\prev#1{#1^{k-1}}
\let\opt\widehat

\def\AlphaSpace{\mathscr{A}}

\def\dir#1{\Delta#1}

%% file: Tikz_image1.tex
\pgfplotslegendfromname{leg:performance}\\
\medskip
\begin{subfigure}[b]{0.45\textwidth}
	\begin{tikzpicture}
		\begin{axis}[%
			width=\linewidth,
			height=0.75\linewidth,
			xmode=log,
			xlabel near ticks,
			ylabel near ticks,
			scaled y ticks=false,
			yminorticks=true,
			minor y tick num=1,
			xminorticks=true,
			minor x tick num=3,
			axis x line*=bottom,
			axis y line*=left,
			legend columns=4,
            legend style={draw=none,font=\small},
			xlabel={computational resources},
			ylabel={$e_{\alpha,\text{rel}}$},
			y tick label style={/pgf/number format/fixed},
			outer sep=0pt,
			font=\footnotesize,
			]

			\addplot[color=Set2-D, line width=1pt] table[x =cputime, y expr=\thisrow{alphaDiff}/2.548]
			{data/bilevelMRIBATAPDPSbrainBlocksparsity015.txt};

			\addplot[color=Set2-B, line width=1pt] table[x =cputime, y expr=\thisrow{alphaDiff}/2.548]
			{data/bilevelMRIBATAPDPSbrainIdsparsity015.txt};

			\addplot[color=Set2-C, line width=1pt] table[x =cputime, y expr=\thisrow{alphaDiff}/2.548]
			{data/bilevelMRIbrainImplicitsparsity015.txt};
		\end{axis}
	\end{tikzpicture}
	\caption{Outer problem - MRI}
	\label{fig:numerical:outer-mri128}
\end{subfigure}%
\begin{subfigure}[b]{0.45\textwidth}
	\begin{tikzpicture}
		\begin{axis}[%
			width=\linewidth,
			height=0.75\linewidth,
			xmode=log,
			xlabel near ticks,
			ylabel near ticks,
			scaled y ticks=false,
			yminorticks=true,
			minor y tick num=1,
			xminorticks=true,
			minor x tick num=3,
			axis x line*=bottom,
			axis y line*=left,
			legend columns=4,
			xlabel={computational resources},
			ylabel={$e_{\alpha,\text{rel}}$},
			y tick label style={/pgf/number format/fixed},
			outer sep=0pt,
			font=\footnotesize,
			legend columns=4,
            legend style={
                draw=none,
                font=\small,
                /tikz/column 2/.style={column sep=1em,},
                /tikz/column 4/.style={column sep=1em,},
                /tikz/column 6/.style={column sep=1em,},
            },
			legend to name = {leg:performance},
			]

			\addplot[color=Set2-D, line width=1pt] table[x =cputime, y expr=\thisrow{alphaDiff}/0.783]
			{data/bilevelDeblurBATAPDPS128block.txt};
			\addlegendentry{PDPS + block-GS}

			\addplot[color=Set2-B, line width=1pt] table[x =cputime, y expr=\thisrow{alphaDiff}/0.783]
			{data/bilevelDeblurBATAPDPS128id.txt};
			\addlegendentry{PDPS + identity}

			\addplot[color=Set2-C, line width=1pt] table[x =cputime, y expr=\thisrow{alphaDiff}/0.783]
			{data/bilevelDeblurImplicit128.txt};
			\addlegendentry{implicit}

			\addplot[color=Set2-A, line width=1pt]
			table[x =cputime, y expr=\thisrow{alphaDiff}/0.783]
			{data/bilevelDeblurTrustRegion.txt};
			\addlegendentry{trust region}

		\end{axis}
	\end{tikzpicture}
	\caption{Outer problem - deblurring}
	\label{fig:numerical:outer-deblur128}
\end{subfigure}%
\\%
\begin{subfigure}[b]{0.45\textwidth}
	\begin{tikzpicture}
		\begin{axis}[%
			width=\linewidth,
			height=0.75\linewidth,
			xmode=log,
			xlabel near ticks,
			ylabel near ticks,
			scaled y ticks=false,
			yminorticks=true,
			minor y tick num=1,
			xminorticks=true,
			minor x tick num=3,
			axis x line*=bottom,
			axis y line*=left,
			xlabel={computational resources},
			ylabel={$e_{u,\text{rel}}$},
			y tick label style={/pgf/number format/fixed},
			outer sep=0pt,
			font=\footnotesize,
			]

			\addplot[color=Set2-D, line width=1pt] table[x =cputime, y expr=\thisrow{u_tilde_diff}/432]
			{data/bilevelMRIBATAPDPSbrainBlocksparsity015.txt};

			\addplot[color=Set2-B, line width=1pt] table[x =cputime, y expr=\thisrow{u_tilde_diff}/432] {data/bilevelMRIBATAPDPSbrainIdsparsity015.txt};

			\addplot[color=Set2-C, line width=1pt] table[x =cputime, y expr=\thisrow{u_tilde_diff}/432]
			{data/bilevelMRIbrainImplicitsparsity015.txt};
		\end{axis}
	\end{tikzpicture}
	\caption{Inner problem - MRI}
	\label{fig:numerical:inner-mri128}
\end{subfigure}%
\begin{subfigure}[b]{0.49\textwidth}
	\begin{tikzpicture}
		\begin{axis}[%
			width=\linewidth,
			height=0.75\linewidth,
			xmode=log,
			ymode=log,
			xlabel near ticks,
			ylabel near ticks,
			scaled y ticks=false,
			yminorticks=true,
			minor y tick num=1,
			xminorticks=true,
			minor x tick num=3,
			axis x line*=bottom,
			axis y line*=left,
			legend columns=4,
			legend style={at={(0.5, 1.0)}, anchor=south,inner sep=0pt,outer sep=0pt,legend cell align=left,align=left,draw=none,fill=none,font=\footnotesize},
			xlabel={computational resources},
			ylabel={$e_{u,\text{rel}}$},
			y tick label style={/pgf/number format/fixed},
			outer sep=0pt,
			font=\footnotesize,
			]

			\addplot[color=Set2-D, line width=1pt] table[x =cputime, y expr=\thisrow{u_tilde_diff}/80.7053]
			{data/bilevelDeblurBATAPDPS128block.txt};

			\addplot[color=Set2-B, line width=1pt] table[x =cputime, y expr=\thisrow{u_tilde_diff}/80.7053] {data/bilevelDeblurBATAPDPS128id.txt};

			\addplot[color=Set2-C, line width=1pt] table[x =cputime, y expr=\thisrow{u_tilde_diff}/80.7053]
			{data/bilevelDeblurImplicit128.txt};

		\end{axis}
	\end{tikzpicture}
	\caption{Inner problem - deblurring}
	\label{fig:numerical:inner-deblur128}
\end{subfigure}%
\\%
\begin{subfigure}[b]{0.45\textwidth}
	\begin{tikzpicture}
		\begin{axis}[%
			width=\linewidth,
			height=0.75\linewidth,
			ymode=log,
			xmode=log,
			xlabel near ticks,
			ylabel near ticks,
			scaled y ticks=false,
			yminorticks=true,
			minor y tick num=1,
			xminorticks=true,
			minor x tick num=3,
			axis x line*=bottom,
			axis y line*=left,
			xlabel={computational resources},
			ylabel={$J\circ S_u + R$},
			y tick label style={/pgf/number format/fixed},
			outer sep=0pt,
			font=\footnotesize,
			]

			\addplot [color=Set2-D, line width=1pt] table[x=cputime,y=JplusR]{data/bilevelMRIBATAPDPSbrainBlocksparsity015.txt};

			\addplot [color=Set2-B, line width=1pt] table[x=cputime,y=JplusR]{data/bilevelMRIBATAPDPSbrainIdsparsity015.txt};

			\addplot [color=Set2-C, line width=1pt] table[x=cputime,y=JplusR]{data/bilevelMRIbrainImplicitsparsity015.txt};

		\end{axis}
	\end{tikzpicture}
	\caption{Outer problem - MRI}
	\label{fig:numerical:outer-value-mri128}
\end{subfigure}%
\begin{subfigure}[b]{0.45\textwidth}
	\begin{tikzpicture}
		\begin{axis}[%
			width=\linewidth,
			height=0.75\linewidth,
			ymode=log,
			xmode=log,
			xlabel near ticks,
			ylabel near ticks,
			scaled y ticks=false,
			yminorticks=true,
			minor y tick num=1,
			xminorticks=true,
			minor x tick num=3,
			axis x line*=bottom,
			axis y line*=left,
			xlabel={computational resources},
			ylabel={$J\circ S_u + R$},
			y tick label style={/pgf/number format/fixed},
			outer sep=0pt,
			font=\footnotesize,
			]

			\addplot [color=Set2-D, line width=1pt] table[x=cputime,y=JplusR]{data/bilevelDeblurBATAPDPS128block.txt};

			\addplot [color=Set2-B, line width=1pt] table[x=cputime,y=JplusR]{data/bilevelDeblurBATAPDPS128id.txt};

			\addplot[color=Set2-A, line width=1pt]
			table[x =cputime, y=JplusR]
			{data/bilevelDeblurTrustRegion.txt};

			\addplot [color=Set2-C, line width=1pt] table[x=cputime,y=JplusR]{data/bilevelDeblurImplicit128.txt};
		\end{axis}
	\end{tikzpicture}
	\caption{Outer problem - deblurring}
	\label{fig:numerical:outer-value-deblur128}
\end{subfigure}%